\documentclass[11pt]{amsart}
\usepackage{amsmath, amsthm, amsfonts, amssymb}
\usepackage{mathrsfs,graphicx}
\providecommand{\noopsort[1]{}}
\usepackage{bbm}
\usepackage{dsfont}
\usepackage{ifthen}
\numberwithin{equation}{section}
\usepackage{a4}
\usepackage{hyperref}
\usepackage{enumerate}

\newtheorem{thm}{Theorem}[section]
\newtheorem{cor}[thm]{Corollary}
\newtheorem{prop}[thm]{Proposition}
\newtheorem{lem}[thm]{Lemma}

\theoremstyle{remark}
\newtheorem{rem}[thm]{Remark}

\newtheorem{hyp}[thm]{Hypothesis}
\newtheorem{example}[thm]{Example}

\theoremstyle{definition}
\newtheorem{defn}[thm]{Definition}

\newcommand{\coloneqq}{\mathrel{\mathop:}=}

\newcommand{\eps}{\varepsilon}

\newcommand{\one}{\mathds{1}}

\newcommand{\R}{\mathds{R}}
\newcommand{\C}{\mathds{C}}

\newcommand{\N}{\mathds{N}}

\newcommand{\expect}{\mathds{E}}
\newcommand*\mE{\mathds{E}}
\newcommand*\me{\mathrm{e}}

\newcommand{\Pb}{\mathds{P}}
\newcommand{\mpr}{\mathds{P}}
\newcommand{\Rd}{{\mathds{R}^d}}
\newcommand{\dx}{\mathrm{d}}
\newcommand{\dy}{{\mathrm d}y}

\newcommand{\dz}{{\mathrm d}z}
\newcommand{\dt}{{\mathrm d}t}

\newcommand{\tv}{\mathsf{TV}}

\newcommand{\cL}{\mathscr{L}}
\newcommand{\cA}{\mathscr{A}}
\newcommand{\cB}{\mathscr{B}}
\newcommand{\cM}{\mathscr{M}}

\usepackage{color}

\newcommand*{\kb}{\textcolor{black}}
\newcommand*{\mk}{\textcolor{black}}
\newcommand*{\kl}{\textcolor{black}}
\newcommand*{\ml}{\textcolor{black}}

\begin{document}
\title{The fractional Laplacian with reflections}
\author[K. Bogdan]{Krzysztof Bogdan}
\address{Faculty of Pure and Applied Mathematics,
Wroc\l aw University of Science and Technology,
Wyb. Wyspia\'nskiego 27, 50-370 Wroc\l aw, Poland}
\email{krzysztof.bogdan@pwr.edu.pl}
\author[M. Kunze]{Markus Kunze}
\address{Universit\"at Konstanz, Fachbereich Mathematik und Statistik, Fach 193, 78357 Konstanz, Germany}
\email{markus.kunze@uni-konstanz.de}
\thanks{The first-named author has been supported through the DFG-NCN Beethoven Classic 3 programme, contract no. 2018/31/G/ST1/02252 (National Science Center, Poland) and SCHI-419/11–1 (DFG, Germany)}

\begin{abstract}
Motivated by the notion of  isotropic $\alpha$-stable L\'evy processes confined, by reflections, to a bounded open Lipschitz set $D\subset \R^d$, we study some related analytical objects. Thus, we construct the corresponding transition semigroup, identify its generator and prove exponential speed of convergence of the semigroup to a unique stationary distribution for large time.
\end{abstract}

\maketitle

\section{Introduction}
\subsection{\kb{Setting, goals, and results}}
Consider the isotropic $\alpha$-stable L\'evy process $(Y_t, t\ge 0)$ in $\Rd$. 
The intensity of jumps for this process is given by $\nu(x,\dy) = c_{d,\alpha}|x-y|^{-d-\alpha}\dy$, which is the integro-differential kernel of the fractional Laplacian on $\Rd$ (for details, see below). Given an open set $D\subset \Rd$, our \kb{interest lies in} a Markov process $(X_t, t\ge 0)$ that \kb{coincides with} $Y$ as long as $Y$ remains within $D$. However, at the time $\tau_D$ of the first exit of $Y$ from $D$, we \kb{intend} to perform a \textit{reflection}: instead of \kb{allowing $X$ to leave $D$ by going to the point} $z\coloneqq Y_{\tau_D}\kb{\in D^c}$, our process $X$ should be \ml{\textit{restarted}} immediately \kb{(at time $\tau_D$)} at a point $y\in D$, without \kb{venturing into} $D^c$.
\kb{The point $y$ is} chosen (randomly) according to a probability measure $\mu(z, \dy)$, which depends on $z$, \kb{otherwise independently of prior events}.
Based on this heuristic description, we expect the intensity of jumps of $(X_t, t\ge 0)$ to be the following integral kernel on $D$,
\begin{equation}\label{e.kgKt}
\gamma(x,\dy)\coloneqq \nu(x,\dy)+\int_{D^c} \nu(x,\dz) \mu(z,\dy).
\end{equation}

Actually, the outcome of this work is not a Markov process, but a \textit{conservative Markovian semigroup} $(\mk{K(t)}, t > 0)$ having $\gamma$ as the kernel of \kb{the} generator. We also prove that $(\mk{K(t)}, t > 0)$ has a unique stationary density and is \textit{exponentially asymptotically stable}. Regarding the construction and analysis of the process \kb{$X$},
we will \kb{cover them} in \cite{KB-MK-Mp}. \mk{Now, let us state our} assumptions on $D$ and $\mu$:
\begin{hyp}\label{hyp1}
Let $D\neq \emptyset$ be an open bounded Lipschitz subset of $\mathbb{R}^d$
\mk{and let $\mu: D^c\times \mathscr{B}(D) \to [0,1]$ have the following properties:
\begin{enumerate}
[(i)]
    \item For every $A\in \mathscr{B}(D)$, the map 
    $z\mapsto \mu(z, A)$ is Borel measurable, and \kb{for every $z\in D^c$,} $\mu(z, \cdot)$ 
    is a Borel probability measure on $D$.
    \item There exists a compact set $H\subset D$
    such that $\vartheta\coloneqq \inf\limits_{z\in D^c}\mu(z, H)>0$.
    \item The map $z\mapsto \mu(z, \cdot)$ is weakly continuous
    at $\partial D$; that is, if points $z_n\in D^c$ converge to
    $z\in \partial D$, then the measures $\mu(z_n, \cdot)$ converge weakly to $\mu(z, \cdot)$.
\end{enumerate}
}
\end{hyp}

\kb{Hypothesis~\ref{hyp1} is assumed throughout, though we  recall it from time to time.
Let us comment on the assumptions. The geometric regularity of $D$ is technically important but can be relaxed; see Remark~\ref{r.rh}. Condition (i) means, in short, that the \textit{repulsion kernel} $\mu$ is a 
\textit{stochastic kernel}. The \textit{lower bound} (ii)
is crucial for Theorem~\ref{t.m} and, informally, for avoiding an infinite number of reflections \kb{within} finite time. The \textit{weak continuity} (iii) 
allows for a description of the corresponding infinitesimal generator and 
    boundary conditions in terms of continuous functions on $\overline{D}$; see Theorem~\ref{t.generator} and Remark~\ref{r.bc}. }

\subsection{\kb{Context and literature}}
Reflections similar to the ones that we study for jump processes appeared first in the work of Feller \cite{feller-diffusion} \kb{concerning} one-dimensional diffusions. Feller coined the \kb{term} \emph{instantaneous return processes} for the resulting 
\kb{Markov} processes.  
On the level of the corresponding generator of the process, which is a second order elliptic differential operator in  \cite{feller-diffusion}, 
\kb{these} reflections lead to certain nonlocal boundary conditions.  Further operators of this form, 
\kb{including those} in higher dimension, have been 
\kb{explored} by various \kb{techniques and authors,} including Galakhov and Skubachevski{\u\i} \cite{galaskub}, Ben-Ari and Pinski \cite{b-ap07}, Arendt, Kunkel, and Kunze \cite{akk16}, and Kunze \cite{kunze20}; see also the monograph of Taira \cite{taira} and the references therein.

\kb{\textit{Boundary conditions} for jump processes and nonlocal operators remain an open subject, although various mechanisms of reflection from $D^c$ and Neumann-type conditions have been proposed in the literature. When reflecting jump processes from $D^c$ to $D$, one faces the choice of making} $X_{\tau_D}$ depend on $Y_{\tau_D-}$, 
on $Y_{\tau_D}$,
or both.
\kb{These scenarios are notably richer than those for diffusions.}
Bogdan, Burdzy, and Chen \cite{MR2006232} propose \kb{both} censored and actively reflected processes, with the reflection depending (deterministically) \kb{solely} on $Y_{\tau_D-}$. For $D$ being the half-space, 
Barles, Chasseigne, Georgelin, and Jakobsen \cite{MR3217703} discuss geometrically motivated reflections that deterministically depend on $(Y_{\tau_D-},Y_{\tau_D})$. 
Dipierro, Ros-Oton, and Valdinoci \cite[p.\ 378]{MR3651008} postulate a random mechanism of reflection \kb{that employs} $\mu(z,\mathrm{d}y)=\nu(z,\mathrm{d}y)/\nu(z,D)$. Notably, the papers \cite{MR3217703, MR3651008} discuss Neumann-type problems, but neither the semigroup nor the corresponding Markov process\kb{, both of which are highly nontrivial to construct and study \cite{MR2006232}.}
\kb{In comparison,} Vondra\v{c}ek \cite{MR4245573} \kb{constructs a Markov process $X$ that returns to $D$ following the distribution $\nu(Y_{\tau_{D}},\mathrm{d}y)/\nu(Y_{\tau_{D}},D)$, 
but the process $X$ remains at $Y_{\tau_D}\in D^c$} for a unit exponential time before \kb{returning} to $D$. 
\kb{This arrangement} helps avoid the \kb{problematic} scenario of an infinite number of 
\kb{passes} between $D$ and $D^c$ in finite time, while also taking \cite{MR4245573} beyond 
the \kb{tentative setting of instantaneous} reflections \kb{in \cite{MR3217703, MR3651008}, and our paper}.

\kb{Reflecting Markov processes may be considered as an instance of \textit{stochastic resetting}, which is a hot topic in statistical physics concerned, among others, with equilibrium distributions, search optimization, renewal theory and modelling; see, e.g., Evans, Majumdar, and Schehr \cite{MR4093464}, Garbaczewski and \.{Z}aba \cite{MR4474283}, and Stanislavsky and Weron \cite{MR4378848}.}
\kb{Specific reflections, \textit{resurrections}, or \textit{recurrent extensions} of stochastic processes with \textit{scaling} are studied by Kim, Song, and Vondra\v{c}ek \cite{MR4520527} for the halfline using the Lamperti transform. See also Rivero \cite{MR2364226} and Fitzsimmons \cite{MR2266714}. Kim, Song, and Vondra\v{c}ek \cite{kim2022potential} explore halfspaces using Dirichlet forms; see also Chen and Song \cite{MR1952456}. Recent results for the entire real line can be found in Pant\'{\i}, Pardo, and Rivero \cite{MR4140082} and Iksanov and Pilipenko \cite{MR4514832}; see also Chaumont, Pant\'{\i}, and Rivero \cite{MR3160562}.}  
\kb{Furthermore, Bobrowski \cite{bobrowski2022concatenation} studies resurrections or \textit{concatenations} on metric graphs via \textit{exit laws} and resolvents. We should add that a general probabilistic approach to concatenation or \textit{piecing-out} of Markov processes was proposed by Ikeda, Nagasawa, and Watanabe \cite{MR202197}, see also Meyer \cite{meyer75}. Sharpe \cite{MR958914} provides an introduction and further references to the method. For a recent presentation and references to concatenation of \textit{right processes}, we refer to Werner \cite{MR4247975}.}

\subsection{\kb{Approach, perspectives, and content}}
\kb{We construct the Markovian transition semigroup corresponding to the kernel $\gamma$ by the method of
nonlocal Schr\"odinger perturbations of integral kernels by Bogdan and Sydor \cite{MR3295773}.
\ml{
In the special case of transition kernels, the method may handle 
analytic aspects of concatenation of Markov processes. Namely,
given a sub-Markovian kernel, a suitable repulsion kernel defines 
a larger transition kernel, possibly Markovian, which 
corresponds to a specific partial differential equation with 
(nonlocal) boundary conditions. For instance, see} 
Corollary~\ref{c.se}.}

\kb{In the present setting we consider a}
class of \kb{repulsion} kernels \kl{$\mu$} \kb{for the Dirichlet heat kernel}
of the  fractional Laplacian. 
On the one hand, the restriction to \ml{the} fractional Laplacian is \kb{dictated merely by the technical convenience and general interest in this operator, so} 
generalizations are quite obvious. On the other hand, \kb{the papers} \kb{\cite{MR2006232, MR3217703, MR3651008, MR4245573, MR4520527, kim2022potential}} make do without the 
\kb{lower bound for the repulsion kernel}
\kl{that}
we assume in Hypothesis~\ref{hyp1}\kb{(ii).} 
\kb{Therefore}
the above references and the present paper \kb{should be considered as}
different ramification\kb{s} of the
problem of constructing operators, semigroups and Markov processes with specific boundary conditions. \kb{As we already mentioned, this}
area of research is motivated by the Neumann-type boundary-value problems 
%
and 
the problem of piecing-out or concatenation  of Markov processes.
Its \kb{objective}, beyond the construction, are the questions of the large-time and boundary behavior of the \kb{resulting} semigroup and process, as well as applications to nonlocal differential equations with those boundary conditions.

\kb{The} paper is organized as follows.
Section~\ref{sec:prelim} \kb{provides essential preliminaries, including an introduction to} the fractional Laplacian \kb{and} related potential theory. \kb{It also introduces} the Dirichlet heat kernel $(\mk{p_D(t)}, t>0)$ of the set $D$.
In Section~\ref{sec:tk}, we \kb{proceed to construct} the kernel $(\mk{k(t)}, t>0)$ \kb{that defines} the semigroup $(\mk{K(t)})$. \kb{We also prove that}
$\int_D \mk{k(t,x,y)}\dy=1$ for all $x\in D$ and $t>0$.
%
Section \ref{s.Lt} \kb{delves into the study of} the resolvent of $(\mk{K(t)})$. 
\kb{In} Section~\ref{s.sg}, we \textit{characterize} the generator of the semigroup and discuss the associated boundary conditions. \kb{Additionally, we provide a solution to a typical boundary value problem.}
In Section~\ref{s.im} we prove the existence of a unique invariant measure (density) and the exponential convergence of the semigroup to the stationary measure for large time. 
\kb{The reader may also find a few examples illustrating our results.
Example~\ref{ex.dirac} and Remark~\ref{r.tough} (\ml{restarting at} a fixed point $x_0\in D$) show that, in general, the semigroup $(K(t))$ is neither symmetric nor does it act on $L^2$. Examples~\ref{ex.1}, \ref{ex.2}, and \ref{ex.3} focus on the boundary conditions induced by $\mu$.
In Remark~\ref{r.iob}, we propose directions for future research and generalizations.}

\kb{As we mentioned, the} construction of the semigroup $(\mk{K(t)})$ 
is purely analytic, based on nonlocal \kb{Schr\"odinger}
perturbation, \kb{or} Duhamel formula applied to $p^D$. 
The reader interested in the \kb{existence and properties of a} Markov process resulting from $(\mk{K(t)})$ is referred to  the forthcoming paper \cite{KB-MK-Mp} by the authors.

\subsection*{Acknowledgments.} We thank \kb{Adam Bobrowski,} Jan van Casteren, \kb{Damian Fafuła, Piotr Garbaczewski,} Wolfhard Hansen, Tadeusz Kulczycki, Tomasz Klimsiak, Tomasz Komorowski, \kb{Andrey Pilipenko, Victor Rivero}, Tomasz Szarek, Paweł Sztonyk \kb{and Zoran Vondra\v{c}ek} for discussions, comments and references. \kb{We thank the referee for insightful comments, including those on Hypothesis~\ref{hyp1}(iii), which strongly influenced the presentation and content of the paper.}

\section{Preliminaries}\label{sec:prelim}
We often use $\coloneqq$ for definitions, \kb{e.g.,} 
$\mathds{N}_0\coloneqq \{0,1,2\ldots\}$ and $\mathds{N}\coloneqq \{1,2,\ldots\}$. \kb{$\one_A$ denotes the indicator function of $A$ and $\one$ is the indicator of the full space.} We let $d\in \mathds{N}$ and consider the Euclidean space $\Rd$.
All  the sets, functions, measures and kernels considered in the paper are Borel.
If not stated otherwise,  functions take values in the extended real line. 
For $x\in \Rd$ and $r\in (0,\infty)$ we denote by $B(x,r)=\{y\in \Rd: |y-x|<r\}$ the ball with radius $r$ and center at~$x$. 
We require integrals  
to be nonnegative or absolutely convergent.

\subsection{Fractional Laplacian}\label{sec:pot-theo-notions}
Let 
$\alpha \in (0,2)$, and 
$$\nu(x)\coloneqq c_{d,\alpha}|x|^{-d-\alpha},\quad x\in \Rd,$$ 
where 
$$c_{d,\alpha} \coloneqq \frac{2^\alpha \Gamma((d+\alpha)/2)}{\pi^{d/2}|\Gamma(-\alpha/2)|}.$$
The constant $c_{d,\alpha}$ is chosen in such a way that
$$
|\xi|^{\alpha}=\int_\Rd (1-\cos \xi\cdot x)\nu(x)\dx x\,, \quad \xi \in \Rd.
$$

According to Fourier inversion and the L\'evy--Khinchine formula, there is a convolution semigroup of smooth probability densities $(\mk{p(t)}, t>0)$ such that
\begin{equation}\label{e.LKf}
\int_\Rd \me^{i\xi\cdot x} \mk{p(t,x)}\dx x=\me^{-t|\xi|^\alpha}, \quad 
\xi\in \Rd.
\end{equation}
It follows that 
\begin{equation}\label{e.sc1}
\mk{p(t, x)=t^{-d/\alpha}p(1,t^{-1/\alpha}x)},\qquad t>0,\ x\in \Rd.
\end{equation}
The above \textit{scaling} implies in particular that 
\begin{equation}\label{e.lbm}
\int_{\{|x|<ct^{1/\alpha}\}}\mk{p(t,x)}\dx x=\int_{\{|x|<c\}}\mk{p(1,x)}\dx x >0,\qquad \kb{c,}\,t>0.
\end{equation}
It is \mk{well known} that $\mk{p(1,x)} \approx (1+|x|)^{-d-\alpha}$ for $x\in \Rd$, see, e.g., \cite[remarks after Theorem~21]{MR3165234} or Kwa\'snicki \cite[(2.11)]{MR3613319}. Here $\approx$ indicates that the ratio of both sides is bounded from above and below by a (strictly positive) constant. We call such comparisons \textit{sharp}. Thus,
\begin{equation}\label{e.sc}
\mk{p(t,x)}\approx t^{-d/\alpha}\wedge \frac{t}{|x|^{d+\alpha}},\qquad t>0,\ x\in \Rd.
\end{equation}

\mk{Setting $p(t, x,y)\coloneqq p(t,y-x)$, $x,y\in \Rd$, $t>0$, we introduce a} (translation-invariant) transition density.
On the space $D([0,\infty))$ of \ml{c\`adl\`ag} functions (paths) $\omega:[0,\infty)\to \Rd$, we define the canonical proces\kb{s} $Y_t(\omega)=\omega_t$, $t\ge 0$\kb{,} and
define Markovian measures $\mpr^x$, $x\in \Rd$, as follows. For (starting points) $x\in \Rd$, (times) $0 < t_1<t_2<\ldots < t_n$ and (windows) $A_1, A_2,\ldots, A_n\subset \Rd$ we let
\begin{align*}
&\mpr^x(\omega_{t_1} \in A_1,\ldots, \omega_{t_n} \in A_n)=\\
&\int\limits_{A_1}\dx x_1 \int\limits_{A_2}\dx x_2 \ldots\int\limits_{A_n}\dx x_n \, \mk{p(t_1, x, x_1)p(t_2-t_1, x_1, x_2) \cdots p(t_n-t_{n-1}, x_{n-1}, x_n).}
\end{align*}
By the Kolmogorov extension theorem\kb{,} these finite\kb{-}dimensional distributions uniquely determine $\mpr^x$, the law of the Markov process $(Y_t)$ starting from $x$. We let $\mE^x$ denote the corresponding expectation. As it turns out, $Y$ is the isotropic $\alpha$-stable process, a specific \textit{symmetric L\'evy process} in $\Rd$ with the L\'evy triplet $(0,\nu,0)$, see, e.g., Sato \cite[Section 11]{MR1739520}. 
To analyze $Y$ we  use the standard complete right-continuous filtration $(\mathcal F_t, t\geq 0)$, see Protter \cite[Theorem I.31]{MR2273672}. In passing we also recall that every L\'evy process is Feller, see \cite{MR2273672} \kb{or} B\"{o}ttcher, Schilling and Wang \cite{MR3156646}, meaning that the operator semigroup 
\begin{equation}\label{eq.semigroup}
\mk{P(t)}f(x)=\mE^xf(Y_t), \quad x\in \Rd,\; t\ge 0,
\end{equation}
leaves $C_0(\Rd)$, \mk{the space of continuous functions
$f:\Rd\to \R$ satisfying $f(x) \to 0$ as $|x|\to \infty$}, invariant and is strongly continuous on that space. 
As in the \kb{I}ntroduction, we \mk{write} $\nu(x,y)=\nu(y-x)=c_{d,\alpha}|y-x|^{-d-\alpha}$ and $\nu(x,\dy)=\nu(x,y)\dy$ and for $u\colon \Rd\to \R$ and $x\in \Rd$ we define
\begin{align}\label{eq:L-def}
\Delta^{\alpha/2}u(x)
&= \lim_{\eps\to 0^+} \int_{\{|y-x|>\eps\}} \!\!\big[u(y)-u(x)\big]\nu(x,\dy)\\
&= \lim_{\eps\to 0^+} \tfrac12 \!\int_{\{|z|>\eps\}} \!\!\big[u(x+z)+u(x-z)-2u(x)\big]\nu(z)\,\dz.\nonumber
\end{align}
This is the fractional Laplacian (it is also common to use the notation $-(-\Delta)^{\alpha/2}$ for this operator). The limit exists, e.g., for $u\in C_c^\infty (\Rd)$, the smooth functions with compact  support. The operator $\Delta^{\alpha/2}$ extends to the infinitesimal generator of the Feller semigroup defined by \eqref{eq.semigroup} on $C_0(\Rd)$. For a discussion of the many equivalent definitions of $\Delta^{\alpha/2}$ we refer to \cite{MR3613319}.\smallskip

\subsection{\kb{Exit times}} 
\kb{The} \textit{time of the first exit} of $Y$ from \kb{an open set $U\subset \Rd$} is
$$\tau_U\coloneqq\inf\{t>0: \, Y_t\notin U\}.$$
By translation invariance and scaling of $p(t, x, y)$, the law of $\{x+Y_t, t\ge 0\}$ under $\mpr^0$ is the same as the law of $\{Y_t, t\ge 0\}$ under $\mpr^x$ for $x\in \Rd$, and, under $\mpr^0$, the law of $\{cY_t\ge 0\}$ equals that of $\{Y_{c^\alpha t}\ge 0\}$. Hence, $\tau_{x+U}$ has the same law under $\mpr^0$ as $\tau_{\ml{U}}$ under $\mpr^x$ and, under $\mpr^0$,
$c\tau_{U}$ has the same law as $\tau_{c^\alpha U}$.

Recall that $D$ is a nonempty bounded open subset of $\Rd$ which is Lipschitz (in the sense of \cite[p.\ 156]{MR1703823}).  In particular, $D$ is regular, meaning that
\begin{equation}\label{e.reg}
\mpr^x(\tau_D=0)=1\quad \kb{\mbox{for}}\quad x\in \partial D.
\end{equation}
This follows from the radial symmetry of $p_t(x)$, Zaremba's exterior cone property of $D$ and Blumenthal's $0$-$1$ law, as in \cite[Section 4.4]{MR2152573}. We also refer to \cite[VII.3, IV]{MR850715} for analytic definitions and treatment of regularity and to connections to the theory of stochastic processes. It is well known that the regularity of $D$ is equivalent to solvability of the Dirichlet problem with arbitrary continuous data; see the above references. \kb{This condition will be important in Section~\ref{sect.dirichlet}.}

The boundedness of $D$ assures that the process leaves $D$ in \emph{finite time}:
$\mpr^x(\tau_D<\infty)=1$, see, e.g., \kb{Bogdan and Byczkowski}  \cite[Subsection 2.3]{MR1671973}. We consider $Y_{\tau_D}$, the position at the exit time, and
$Y_{\tau_D-}\coloneqq \lim_{s\uparrow \tau_D} Y_s$, the position just before the exit. 
Thanks to the Lipschitz geometry of $D$,
\begin{equation}\label{e.nhtb}
\mpr^x(Y_{\tau_D}\in \partial D)=0 \quad \mbox{for}\quad  x\in D.
\end{equation}
\kb{This is the principle of ``not hitting the boundary upon the first exit''} known since Bogdan \cite[Lemma~6]{MR1438304}; see \kb{also} Bogdan, Grzywny, Pietruska-Pa\l{}uba and Rutkowski \cite[Corollary A.2]{BOGDAN2019} for generalizations. 
\kb{In particular,} the first exit from $D$ occurs by a \textit{jump}\kb{:}
\begin{equation}\label{e.nhb}
\mpr^x[\tau_D<\infty, Y_{\tau_D-}\neq Y_{\tau_D}]=1,\qquad x\in D.
\end{equation}
The random variable $\tau_D$ leads to important analytic objects\kb{, especially}
\[
\mk{p_D(t,x,y) \coloneqq p(t,x,y) - \mE^x[p(t-\tau_D, Y_{\tau_D},y)};\, \tau_D<t],\quad t>0,\ x,y\in D,
\]
the \textit{Dirichlet heat kernel}, 
see, e.g., Chung and Zhao \cite[Chapter 2.2]{MR1329992}. 
The function is the transition density of the process $Y$ \emph{killed} upon exiting $D$. \kb{Namely,}
\begin{equation}\label{e.kp}
\mE^x[f(Y_t); t<\tau_D]=\int_D f(y)\mk{p_D(t,x,y)}\dy, \quad x\in D,\, t>0\,,
\end{equation}
and the following Chapman--Kolmogorov equations hold for $p^D$:
\begin{equation}\label{e.CKpD}
\int_D \mk{p_D(s, x,z)p_D(t,z,y)dz=p_D(t+s,x,y)}\,,\quad s,t>0 ,\,
x,y\in D\,.
\end{equation}
It is also \mk{well known} that $\mk{p_D(t,x,y)}$ is jointly continuous and positive for all
$(t,x,y)\in (0,\infty)\times D\times D$.
\kb{In view of} 
\eqref{e.nhb}, the joint distribution of $(\tau_D, Y_{\tau_D-}, Y_{\tau_D})$ 
under $\mpr^x$ is given by the following \emph{Ikeda--Watanabe formula}
\begin{align}\label{eq:IW}
\mpr^x[\tau_D\in I,\, Y_{\tau_D-}\in A,\, Y_{\tau_D}\in B]=
\int\limits_I \!\dx s \int\limits_{A} \!\dx v \int\limits_B\! \dx z \, \mk{p_D(s,x,v)}\nu(v,z),
\end{align}
where $x\in D$, $I\subset [0,\infty)$, $A\subset D$ and $B\subset D^c$, see, e.g., Bogdan, Rosiński, Serafin and Wojciechowski \cite[S\kb{ubs}ection~4.2]{MR3737628}.
\kb{Let}
\begin{equation}\label{eq:Pk}
\mk{h_D(x,z)}\coloneqq
\int_0^\infty \dx s \int_{D}\dx v\, \mk{p_D(s,x,v)}\nu(v,z), 
\quad x\in D,\ z\in D^c.
\end{equation}
\kb{By \eqref{eq:IW}, the}
function is the density of the \textit{harmonic measure} of $D$\kb{:}
\begin{equation}\label{eq:hm}
\mpr^x(Y_{\tau_D}\in B)=\int_B \mk{h_D(x,z)}\dx z, \qquad x\in D,\quad B\subset D^c.
\end{equation}
By the same reason,
\begin{equation}\label{eq:1}
\int_0^\infty \dx s \int_{D}\dx v \int_{D^c} \dx z \, \mk{p_D(s,x,v)}\nu(v,z) =1,\qquad x\in D.
\end{equation}
The \textit{survival probability} $\Pb^x( \tau_D > t)$ can be expressed in two ways as follows:
\begin{align}\label{e.tfsp1}
\Pb^x( \tau_D > t)&= \int_t^\infty \dx s \int_D \dx v \int_{D^c} \dx z\,  \mk{p_D(s, x,v)}\nu (v,z)\\
&= \int_D  \mk{p_D(t, x,y)}\, \dx y
,\quad t>0,\, x\in D.\label{e.tfsp2}
\end{align}
Indeed, the first equation follows from the Ikeda--Watanabe formula \eqref{eq:IW}, and the second from \eqref{e.kp}. 
Combining this with \eqref{eq:1} yields, for all $t>0$, $x\in D$,
\begin{equation}\label{eq:a1}
\int_D \mk{p_D(t,x,y)}\dx y+
\int_0^t\dx s \int_{D}\dx v \int_{D^c} \dx z\, \mk{p_D(s, x,v)}
\nu(v,z)=1.
\end{equation}
\kb{
We also have
\begin{equation}\label{e.fsp}
\Pb^x( \tau_D \le t)= \int_0^t \dx s \int_D \dx v \int_{D^c} \dx z\,  p_D(s, x,v)\nu (v,z),\quad t>0,\, x\in D.   
\end{equation}
}
%
%
For future use, we record the following fact on the \kb{survival probability}.
\begin{lem}\label{l.uniformstopping}
If $\kb{F}
\subset D$ is compact and $T\in (0,\infty)$, then a constant $\eta=\eta(\kb{F},
T)>0$ exists, such that $\mpr^x( \tau_D > t)\geq \eta$ for all $x\in \kb{F}
$ and $0<t\le T$. 
\end{lem}
\begin{proof}
Let $r=\textrm{dist}(\kb{F},D^c)$. Of course, $0<r<\infty$. 
We have $\mpr^x( \tau_D > t)\ge \mpr^x( \tau_{B(x,r)} > t)=\mpr^0( \tau_{B(0,r)} > t)$. The latter is clearly nonincreasing in $t$. It is also strictly positive, see, e.g., Chen and Song \cite[Theorem 2.4]{MR1473631}.
\end{proof}
\kb{For clarity we} 
note 
that some of the arguments in \cite[Theorem 2.4]{MR1473631} refer to Chung and Zhao \cite{MR1329992}, who deal with the Brownian motion, but the arguments apply more generally. \kb{Alternatively, we may use the}
sharp explicit bounds for $\mpr^x( \tau_{B(0,r)} > t)$ 
given in Bogdan, Grzywny, Ryznar \cite[Lemma 6]{MR2722789}.

\subsection{The killed semigroup}\label{sec:8.3}

In this section, we consider $(\mk{P_D(t)}, \ml{t > 0})$, the semigroup 
of the process killed upon leaving $D$. 
\kb{Thus,}
\begin{equation}\label{eq.killedrep}
\mk{P_D(t)}f(x) \kb{\coloneqq} \expect^x\big[ f(Y_t)\one_{\{ t< \tau_D\}}\big] = \int_{D} f(y)\mk{p_D(t,x,y)} \dx y,\quad x\in D.
\end{equation}
\kb{It is} 
a sub-Markovian semigroup on the space $B_b(D)$ of bounded measurable functions on $D$. 
\mk{By $C_0(D)$ we denote the space of continuous functions on
\kb{$D$} that vanish \kb{at} $\partial D$. Recall that a semigroup
$S(t)$ of kernel operators on $B_b(D)$ is called \emph{Feller semigroup} if $S(t)C_0(D) \subset C_0(D)$, $t\geq 0$, and for every
$f\in C_0(D)$ the \emph{orbit} $t\mapsto S(t)f$ is continuous
on $[0,\infty)$ in the supremum norm $\|\cdot\|_\infty$.
\ml{It} is called \emph{strong Feller} if $S(t)B_b(D)
\subset C_b(D)$, $t>0$, where $C_b(D)$ refers to the space of
bounded, continuous functions on $D$. \ml{$(S(t))$} is called
\emph{$C_b$-Feller} if $S(t)f\to f$ on compact subsets of $D$
when $t\to 0$ and $f\in C_b(D)$, see Definition \ref{d.cbsemigroup}
for details. Concerning the semigroup of the killed process, we
have:}

\begin{lem}\label{l.pd}
The semigroup $(\mk{P_D(t)})$ is Feller, strong Feller, and $C_b$-Feller.
Moreover, $\mk{P_D(t)}B_b(D) \subset C_0(D)$ for every $t>0$.
\end{lem}
\begin{proof}
\kb{Due to \cite[page 68]{chung86}, the} semigroup $\mk{P_D}$ is a Feller semigroup and enjoys the strong Feller property, because these properties hold true for the unkilled semigroup $P$ and $D$ is regular. At this point \cite[\mk{Lemma} 3.1]{sch98} implies that $\mk{P_D}$ is a $C_b$-semigroup.
\mk{Note that \cite{sch98} is actually
concerned with Feller semigroups on $\R^d$, but
the proof of \cite[Lemma 3.1]{sch98} applies to
open subsets 
of $\Rd$, \kl{too}.}
\kb{The last statement of Lemma~\ref{l.pd} can be proved as in \cite[Theorem 2.7]{MR1329992}.}
\end{proof}

\kb{For our development in Section~\ref{s.sg}, we need to}
characterize the $C_b$-generator of the $C_b$-Feller semigroup $\mk{P_D}$\kb{;} see \ref{d.cbgenerator} for the definition. 
The operator is a (typically strict) extension of the generator of the Feller semigroup on $C_0(D)$. 
Whereas the latter is defined as the derivative at $t=0$ of the orbits with respect to the norm $\|\cdot\|_\infty$, the $C_b$-generator is defined via the Laplace transform of the semigroup. 
\kb{To this end we will use a general} Theorem \ref{t.generatorchar}, \kl{providing} several equivalent characterizations of 
$C_b$-generator\kb{s}\kl{, }
\kb{and the following \kb{simple} lemma.}

\begin{lem}\label{l.stopestimate}
There exists $c=c(D,\alpha)>0$ such that if $D\subset B(0,R)$, then 
\[
\Pb^x (\tau_D \leq t ; Y_{\tau_D} \in B(0,2R)^c) \leq c R^{-\alpha } \Pb^x (\tau_D \leq t),\qquad x\in D,\; t>0.
\]
\end{lem}

\begin{proof}
We can find a constant $c_1$ such that $\nu (y, B(0,2R)^c) \leq c_1 R^{-\alpha}$ for $y\in B(0,R)$. By the Ikeda--Watanabe formula
\eqref{eq:IW},
\begin{align}
\Pb^x (\tau_D \leq t, Y_{\tau_D} \in B(0,R)^c) & = \int_0^t\dx s \int_D\dx z \, \mk{p_D(s,x,y)}\nu(y, B(0,R)^c) \notag\\
& \leq c_1 R^{-\alpha} \int_0^t\dx s \int_D\dx y \, \mk{p_D(s,x,y)}.\label{eq.stop1}
\end{align}
On the other hand, \kb{by \eqref{e.fsp},}
\begin{align}
\Pb^x(\tau_D \leq t ) 
& \geq c_2 \int_0^t\, ds \int_D\, \dx y \, 
\mk{p_D(s, x,y)},\label{eq.stop2}
\end{align}
where $c_2 \coloneqq \inf_{y\in D} \nu (y, D^c)>0$. The result
follows with $c=c_1/c_2$.
\end{proof}

The final ingredient to characterize the  $C_b$-generator $\cA^D$ of $\mk{P_D}$  in terms of the (unkilled) semigroup $P=\mk{P_\Rd}$ is the
\emph{Dynkin operator} $\mathscr{D}$, defined 
\kb{as}
\begin{equation}\label{eq.dynkinoperator}
\mathscr{D}u(x) \kb{\coloneqq} \lim_{x\to 0} \frac{\expect^x u(Y_{\tau_{B(x,r)}}) - u(x)}{\expect^x \tau_{B(x,r)}},
\end{equation}
for any function $u \in C_b(\Rd)$ and $x\in \Rd$, for which the limit exists. We also denote by $\tilde u$ the extension of a function $u\in C_0(D)$ to $\R^d$ by zero.

\begin{prop}\label{p.killedgenerator}
Let $u,f \in C_b(D)$. Then the statement $u\in D(\cA^D)$ and $\cA^Du =f$ is equivalent to each of the following:
\begin{enumerate}
[{\rm(i)}]
\item $\sup_{t\in (0,1)} \|t^{-1}(\mk{P_D(t)}u - u)\|_\infty <\infty$ and
\begin{equation}
\label{eq.killedgen}
f(x) = \lim_{t\to 0} \frac{\mk{P_D(t)}u(x) - u(x)}{t}, \qquad x\in D.
\end{equation}
\item $u\in C_0(D)$ and 
\begin{equation}
\label{eq.killedgen2}
f(x) = \lim_{t\to 0}\frac{\mk{P(t)}\tilde u(x) - \tilde u(x)}{t}, \qquad x\in D.
\end{equation}
\item $u\in C_0(D)$ and 
\begin{equation}
\label{eq.killedgen3}
f(x) = \lim_{\eps\to 0^+} \int_{\{|y-x|>\eps\}} \big[\tilde u(y) - \tilde u(x)\big] \nu (x, y) \dx y, \qquad x\in D.
\end{equation}
\end{enumerate}
\end{prop}
\begin{proof}
That $u\in D(\cA^D)$ and $\cA^Du=f$ is equivalent to (i) is an immediate consequence of \mk{the equivalence of (i) and (iii) in} Theorem \ref{t.generatorchar}.
We shall prove (iii) $\Rightarrow$ (ii) $\Leftrightarrow$ (i) $\Rightarrow$ (iii). In the proof of (i) $\Leftrightarrow$ (ii), we follow the ideas of \cite[Theorem 2.3]{blm18} which, however, is concerned with the space $C_0(D)$ instead of $C_b(D)$.\smallskip

(i) $\Rightarrow$ (ii). By Lemma \ref{l.pd} we have $\mk{P_D(t)}f \in C_0(D)$ for all $t>0$. 
It follows that the Laplace transform of $\mk{P_D}$ takes values in $C_0(D)$ which, in turn, implies that $D(\cA^D)\subset C_0(D)$. To prove \eqref{eq.killedgen2}, we compare the difference quotients in \eqref{eq.killedgen} and \eqref{eq.killedgen2}. Arguing as in the proof of \cite[Theorem 2.3]{blm18}, we see that
\begin{align*}
& \big[\mk{P_D(t)}u(x) - u(x)\big] - \big[\mk{P(t)}\tilde u (x) - \tilde u (x)\big]=\mk{P_D(t)} \tilde u(x)-\mk{P(t)} 
\tilde u(x)\notag\\
 &=  \, -\expect^x\big[\mk{P(t-\tau_D)} \tilde u(Y_{\tau_D});\tau_D\le t\big]=
  \expect^x \big[ \tilde u(Y_{\tau_D}) - \mk{P(t-\tau_D)}\tilde u(Y_{\tau_D}); \tau_D \leq t\big]\notag\\
 &= \,\expect^x \big[ \tilde u(Y_{\tau_D}) - \mk{P(t-\tau_D)}\tilde u(Y_{\tau_D}); \tau_D \leq t,\
 Y_{\tau_D} \in B(0,2R)\big]\notag\\ 
 &  \quad +
 \expect^x \big[ (\tilde u(Y_{\tau_D}) - \mk{P(t-\tau_D)}\tilde u(Y_{\tau_D})); \tau_D \leq t,\
 Y_{\tau_D} \in B(0,2R)^c\big]\notag\\
 &=:  \,I_1 + I_2.
\end{align*}
Given $\eps>0$, we may pick $R$ so large that $D \subset B(0,R)$ and $R^{-\alpha} \leq \eps$. \kb{If $s\to 0$ then}  
$\mk{P(s)}\tilde u \to \tilde u$ 
uniformly on $\Rd$, 
in particular on $B(0,R)$. We may thus pick $t_0>0$ so that $|I_1| \leq \eps \Pb^x(\tau_D\leq t)$ for all $t\leq t_0$. As for $I_2$, we infer from Lemma \ref{l.stopestimate} that
\[
|I_2| \leq 2\|u\|_\infty c R^{-\alpha} \Pb^x(\tau_D\leq t).
\]
By the choice of $R$, we see that
\[
|I_1 + I_2| \leq C\eps \Pb^x(\tau_D\leq t)
\]
\mk{for $t\leq t_0$. Next pick $r>0$ such that $B(x,r) \subset D$ and note that $\Pb^x(\tau_D \leq t) \leq \Pb^x(\tau_{B(x,r)} \leq t)$.
By \cite[Proposition 2.27(d) and Theorem 5.1]{MR3156646} or  L\'evy inequality and \cite[Remark 1]{MR3350043}, we have $\Pb^x(\tau_{B(x,r)}\leq t) \leq Mt$
for some $M\in (0,\infty)$ and all $t\leq t_0$. Therefore for all such $t$,}
\[
\mk{\frac{|P_D(t)u(x) - P(t)\tilde{u}(x)|}{t} \leq MC\eps}. 
\]
\mk{Since $\eps>0$ was arbitrary, the difference quotient in \eqref{eq.killedgen2} has the same limit as that in \eqref{eq.killedgen}.\smallskip}

(ii) $\Rightarrow$ (i). Since $(\lambda - \cA^D)^{-1}$ exists for $\lambda >0$ and maps $D(\cA^D)$ onto $C_b(D)$, this can be proved as in \cite[Theorem 2.3]{blm18}.\smallskip

(i) $\Rightarrow$ (iii). Adding an isolated point $\dagger$ as a cemetery state to $D$, we can consider the \textit{stopped process} $(Y_{t\wedge \tau_D})_{t\geq 0}$ as a Markov process with state space $E= D\cup\{\dagger\}$. Here $Y_{t\wedge \tau_D}=\dagger$ for
$t\geq\tau_D$. The transition semigroup of this process is $\mk{P_D}$, \kb{if} we extend a function $g$ in $C_b(D)$ to $E$ by setting $g(\dagger)= 0$. 
It follows from the implication (iii) $\Rightarrow$ (ii) in Theorem \ref{t.generatorchar}, that for $u,f$ as in (i) the pair $(u,f)$ belongs to the \emph{full generator} in the sense of \cite[Equation (5.5) of Chapter 1]{ek}. By \cite[Proposition 4.1.7]{ek}, the process
\[
u(Y_{t\wedge \tau_D}) - u(x) - \int_0^{t\wedge \tau_D}f(Y_s)\dx s
\]
is a martingale with respect to $\Pb^x$. Now consider the stopping time $\tau_{B(x,r)}$, where $r>0$ is so small that $B(x,r) \subset D$. Then $\tau_{B(x,r)}\wedge\tau_D=\tau_{B(x,r)}$.
Noting that $u$ and $f$ are bounded functions, it follows from optional stopping that
\[
\expect^x u(Y_{B(x,r)}) - u(x) = \expect^x \int_0^{\tau_{B(x,r)}} f(Y_s)\dx s.
\]
Then,
from the discussion of scaling in Subsection~\ref{sec:pot-theo-notions}, we get
\begin{align*}
&\frac{\expect^x u(Y_{\tau_{B(x,r)}}) - u(x)}{\expect^x \tau_{B(x,r)}} = \frac{1}{\expect^x \tau_{B(x,r)}} \expect^x \int_0^{\tau_{B(x,r)}} f(Y_s)\, \dx s\\
&=\frac{1}{\expect^x \tau_{B(x,1)}} \expect^x r^{-\alpha}\int_0^{r^\alpha\tau_{B(x,1)}} f(Y_s)\, \dx s=\frac{1}{\expect^x \tau_{B(x,1)}} \expect^x \int_0^{\tau_{B(x,1)}} f(Y_{r^\alpha t})\, \dx t\\
&\to\frac{1}{\expect^x \tau_{B(x,1)}} \expect^x \tau_{B(x,1)} f(x)\, \dx s=f(x)\quad \mbox{ as } r\to 0,
\end{align*}
because $f$ is continuous and bounded and we can use the dominated convergence theorem.
This shows that $\mathscr{D}u(x) = f(x)$, for all $x\in D$. At this point \cite[Lemma 3.3]{MR3613319} yields (iii).\smallskip

(iii) $\Rightarrow$ (ii). This follows from \cite[Lemma 3.4]{MR3613319}.
\end{proof}

\subsection{The Dirichlet problem}
\label{sect.dirichlet}

\mk{We are} 
interested in the \kb{following} operators
\begin{equation}\label{eq:1mu}
(\mk{H_D(\lambda)} g)(x)\coloneqq \int_0^\infty \dx s \int_{D}\dx v \int_{D^c} \dx z \,e^{-\lambda s} p_s^D(x,v)\nu(v,z)g(z),
\end{equation}
where \kb{$x\in D$}, $\lambda \geq 0$, and $g$ is a nonnegative or integrable function on $D^c$.
From the Ikeda--Watanabe formula it is immediate that 
\begin{equation}\label{r.fpdl}
H_D(\lambda) g (x) = \expect^x \big[ e^{-\lambda \tau_D}g(Y_{\tau_D}) \big], \quad x\in D.
\end{equation}

\begin{lem}\label{l.philambda_boundary-g}
If $g \in B_b(D^c)$ then $\mk{H_D(\lambda)} g\in C_b(D)$. If $g$ is also continuous at $\partial D$, then $\mk{H_D(\lambda)}g$ extends continuously to $\overline D$ and the extension equals $g$ on $\partial D$.
\end{lem}

\begin{proof}
The result is well known, but the following argument is of some interest.
We have $\mk{H_D(0)\one_{D^c}} =\kb{\one_D}$
on $D$ by \eqref{eq:1}. Therefore, as a consequence of the continuity \mk{(with respect to $x\in D$)} of the integrand \mk{in \eqref{eq:1mu} with $g\equiv 1$ and $\lambda = 0$}, by Vitali's convergence theorem (see, e.g., \cite[Chapter~22]{MR3644418}), the integrand is uniformly integrable for $x$ in every compact subset of $D$. By majorization, the integrand in \eqref{eq:1mu} \mk{for general $g$ and $\lambda$} is also uniformly integrable, \mk{so} $\mk{H_D(\lambda)} g\in C_b(D)$. 

Now assume that $g$ is continuous at $\partial \mk{D}$ and let $x_0\in \partial D$ and $x\in D$.
Taking  \eqref{eq:1} into account again, we find
\begin{align*}
|\mk{H_D(\lambda)}&  g(x) - g(x_0)|  \leq  \int_0^\infty \dt  \int_D \dx v \int_{D^c} \dx z \, \big(1-e^{-\lambda t}\big) 
\mk{p_D(t, x,v)}\nu (v,z) |g(x_0)|\\
& \qquad + \int_0^\infty \dt  \int_D \dx v \int_{D^c} \dx z \,  e^{-\lambda t} \mk{p_D(t, x,v)}\nu (v,z) \big| g(z) -g(x_0)\big|\\
& \leq \|g\|_\infty \int_0^\infty \dt  \int_D \dx v \int_{D^c} \dx z \, \big(1-e^{-\lambda t}\big) \mk{p_D(t,x,v)}\nu (v,z)\\
& \qquad + \int_0^\infty \dt  \int_D \dx v \int_{D^c} \dx z \,  \mk{p_D(t, x,v)}\nu (v,z) \big| g(z) - g(x_0)\big|\\
&  =\colon I_1(x) + I_2(x).
\end{align*}
For arbitrary $\delta>0$ we get
\begin{align*}
I_1(x) & = \|g\|_\infty \, \expect^x \big[ 1- e^{-\lambda \tau_D}\big]
\leq \|g\|_\infty \big[ \mathds{P}^x (\tau_D>\delta) + (1-e^{-\lambda \delta})\big].
\end{align*}
Recall that $D$ \kb{is regular}, whence
$\limsup_{x\to x_0} \mathds{P}^x(\tau_D >\delta) =0$, see \cite[(9)]{chung86} or \cite[Proposition 1.19]{MR1329992}. So,
\[
\limsup_{x\to x_0} I_1(x) \leq \|g\|_\infty (1-e^{-\lambda \delta}).
\]
Since $\delta>0$ was arbitrary, $I_1(x) \to 0$ as $x\to x_0$.
We also have $
I_2(x) = \expect^x \big| g(X_{\tau_D}) - g(x_0)\big| \to 0$ as $x\to x_0$,
since $D$ is regular for the Dirichlet problem, see the discussion following \eqref{e.reg}.
The proof is complete.
\end{proof}

\begin{lem}\label{l.harmonic}
If $\lambda\ge 0$, $g\in \mk{B_b(D^c)}$ \mk{is continuous at $\partial D$}
and $h= \mk{H_D(\lambda)} g$, then  $h\in D(\cA^D)$ and $\cA^Dh = \lambda h$.
\end{lem}
\begin{proof}
As in the proof of Proposition \ref{p.killedgenerator}, we make use of the Dynkin operator $\mathscr{D}$, defined by \eqref{eq.dynkinoperator}.
Fix $x\in D$ and let $r>0$ be so small that $B(x,r) \subset D$. To obtain a function defined on the whole of $\R^d$, we extend $h$ by setting
$h(x) = g(x)$ for $x\in D^c$. By Lemma \ref{l.philambda_boundary-g}, this yields a continuous function on \mk{$\overline{D}$}.
By the strong Markov property, 
\[
\mk{H_{B(x,r)}(\lambda)}h(x) = \mk{H_{B(x,r)}(\lambda)H_D(\lambda)} g(x) = \mk{H_D(\lambda)} g(x) = h(x).
\]
Thus,
\begin{equation}\label{eq.dynkin1}
\lim_{r\to 0} \frac{\mk{H_{B(x,r)(\lambda)}}h(x) - h(x)}{\expect^x\tau_{B(x,r)}} = 0.
\end{equation}
On the other hand, from the discussion in Subsection~\ref{sec:pot-theo-notions}, 
\begin{align}
&\frac{\mk{H_{B(x,r)}(0)}h(x) - \mk{H_{B(x,r)}(\lambda)} h(x)}{\expect^x\tau_{B(x,r)}}  = \frac{\expect^x \big[(1-e^{-\lambda \tau_{B(x,r)}})h(Y_{\tau_{B(x,r)}})\big]}{\expect^x \tau_{B(x,r)}} \notag\\
& = \frac{1}{\expect^0 \tau_{B(0,1)}} \expect^0 \big[r^{-\alpha}(1-e^{-r^\alpha\lambda \tau_{B(0,1)}})h(x+rY_{\tau_{B(0,1)}})\big]\notag\\
&\to \lambda h(x),\quad \mbox{ as } r\to 0,\label{eq.dynkin2}
\end{align}
because $h$ is continuous by Lemma~\ref{l.philambda_boundary-g} and bounded, so we can use the dominated convergence theorem.
Adding \eqref{eq.dynkin1} and \eqref{eq.dynkin2}, we get
$\mathscr{D}h(x) = \lambda h(x)$ and the claim follows from \cite[Lemma 3.3 and 3.4]{MR3613319}, including $h \in D(\cA^D)$.
\end{proof}

\section{The transition kernel with reflections}\label{sec:tk}

In this section we define the transition kernel $\mk{k(t)}$
aforementioned in the \kb{I}ntroduction.
To this end we 
\kb{define} the kernel $\varphi: (0,\infty)\times D \times \mathscr{B}(D)$ 
by
\begin{equation}\label{eq.phi}
\varphi(t,x,A) \coloneqq \int_D\dx v\int_{D^c}\dx z\,  
\mk{p_D(t,x, v)}\nu (v, z)\mu (z,A).
\end{equation}
Let us give an informal interpretation of $\varphi$. We consider $x\in D$ as the value of $Y_0=X_0$, i.e., the starting point of both the processes; $t$ as the value of $\tau_D$, the first exit time of $Y$ from $D$. By the Ikeda--Watanabe formula, $\int_D\dx v \, p_D(t,x,v)\nu(v,z)$ is the density function of 
$(\tau_D, Y_{\tau_D})$. 
If the process $X$ exists as described in the Introduction,
 then $\varphi(t,x,\dx w)\dx t$ is bound to be the joint distribution of $(\tau_D, X_{\tau_D})$. Of course, we mention $X$ and $Y$ only to develop intuition --- in this paper we merely construct a specific transition density $\mk{k(t)}$, using the analytic data: \mk{$p_D(t,x,y)$}, $\nu(x,y)$ and $\mu(z,\cdot)$, but we do not analyze the process $X$, for which see \cite{KB-MK-Mp}.

We note some simple properties of the kernel $\varphi$.
\begin{lem}\label{l.phi}
Let $\varphi$ be defined by \eqref{eq.phi}. Then:
\begin{enumerate}
[{\rm (a)}]
\item For every $x\in D$, we have 
\[ \int_0^\infty \dt \,  \varphi (t, x, D) = 1.\]
\item For every $x\in D$ and $t>0$, we have 
\[
p_D(t, x,D)+
\int_0^{t}\dx s \,\varphi (s, x, D) = 1.
\]
\end{enumerate}
\end{lem}
\begin{proof}
Part (a) follows from Tonelli's theorem, \mk{the fact that every $\mu(z, \cdot)$ is a probability measure}, and \eqref{eq:1}. The proof of  (b) is similar, using \eqref{eq:a1} instead of \eqref{eq:1}.
\end{proof}

We shall use the operator $S$, defined as follows.
For $f: (0,\infty)\times D \to [0,\infty]$,
\[
Sf(t,x) \coloneqq \int_0^t \dx s \int_D
\varphi(s, x, \dx w)f(t-s,w),\quad t>0, \quad x\in D.
\]
By Lemma~\ref{l.phi},
\begin{equation}\label{e.1ip}
0\le S\one(t,x)\le 1, \quad t>0,\quad x\in D.
\end{equation}

If $f: (0,\infty)\times D \times D \to [0,\infty]$, then we slightly abuse notation by also defining
\[
Sf(t,x,y) \coloneqq \int_0^t \dx s \int_D 
\varphi(s, x, \dx w)f(t-s,w,y),\quad t>0,\quad x,y\in D.
\]
We shall apply $S$ to $f(t, x,y)=\mk{p_D(t,x,y)}$.
Given the interpretation of $\varphi$, 
\[Sp_D(t,x,y)\coloneqq \int_0^t \dx s \int_D 
\varphi(s, x, \dx w)\mk{p_D(t-s,w,y)},\quad t>0,\quad x,y\in D,
\]
 tentatively introduces a single reflection from $D^c$ before time $t$.
To accommodate more reflections,
we iterate $S$ and define
\begin{equation}\label{e.dky}
\mk{k(t,x,y)}
\coloneqq \sum_{n=0}^\infty S^n\mk{p_D(t,x,y)}
\end{equation}
for $t>0$, $x,y\in D$. Here, as usual, $S^0\kb{\coloneqq I}$,
the identity.

\mk{Given a} set $A\subset D$ \mk{and a function $f: (0,\infty)\times D\times D \to [0,\infty)$} 
we let $f(t,x,A)=\int_A f(t,z,y)dy$. Then Tonelli's theorem gives
\[
Sf(t,x,A) = \int_0^t \dx s \int_D
\varphi(s, x, \dx w)f(t-s,w,A),\quad t>0,\quad x\in D.
\]
\mk{Applying this to \eqref{e.dky} and using Tonelli's theorem again, yields
\begin{equation}\label{e.dk}
k(t,x,A) \coloneqq \sum_{n=0}^\infty S^np_D(t,x,A).
\end{equation}
}

Being a series of positive functions, $k$ is \mk{well defined}, with values in $[0,\infty]$. 
We also have the following Duhamel (or perturbation)
formula:
\begin{equation}\label{e.pf}
\mk{k(t,x, A) = p_D(t, x, A)} + Sk(t,x, A),\quad t>0,\quad x\in D,\, A\in \mathscr{B}(D).
\end{equation}
We shall gradually prove that $k$ is a transition probability density.
We first establish the Chapman--Kolmogorov equations.
\begin{lem}\label{l.ck}
For every $t,s>0$, $x\in D$ and $A\in \mathscr{B}(D)$, we have
\[
\int_D \, \mk{k(t, x,\dx y) k(s,y,A) = k(t+s, x,A)}.
\]
\end{lem}

\begin{proof} The equality may be obtained as in Bogdan, Hansen and Jakubowski \cite[Lemma~2]{MR2457489}. In fact, it is a special case of Bogdan and Sydor \cite[Lemma~3]{MR3295773}, with the transition kernel $k(s,x,t,A)$ there equal to $\mk{p_D(t-s,x,A)}$ and the perturbing kernel $J(u,z,\dx u_1 \dx z_1)$ given by $\kb{V}(\kb{z},\dx z_1)\one_{(u,\infty)}(u_1)\dx u_1$, where  $\dx u_1$ is the Lebesgue measure on $\R$ and $\kb{V}(x,A)\coloneqq \int_{D^c}\nu(x,\dz)\mu(z,A)$.
\end{proof}

We next prove that $k$ is sub-Markovian. 
Recall  that 
by Lemma \ref{l.phi}(b),
\begin{equation}\label{e.1}
1=\mk{p_D(t,x,D)}+S\one(t,x), \quad t>0,\, x\in D.
\end{equation}

\begin{lem}\label{l.submarkovian}
For all $t>0$ and $x\in D$ we have $\mk{k(t,x,D)} \leq 1$.
\end{lem}

\begin{proof}
If $f : (0,\infty) \times D \to [0,1]$,
then by Lemma \ref{l.phi}(b),
\begin{align*}
0\le \mk{p_D(t,x,D)} + S f (t,x) \leq 
\mk{p_D(t,x,D)} + S \one(t,x)
= 1.
\end{align*}
Since $\mk{p_D(t,x,D)}\leq 1$, it follows by induction that
\[
\sum_{k=0}^{n+1} S^k \mk{p_D}(t,x,D) = \mk{p_D(t,x,D)} + S \big(\sum_{k=0}^n S^k \mk{p_D}\big)(t,x,D) \leq 1.
\]
By \eqref{e.dk} and letting $n\to \infty$, we verify the claim.
\end{proof}

By iterating \eqref{e.1} 
we obtain the identity
\begin{equation}\label{e.11}
1=\mk{p_D(t,x,D)}+S \mk{p_D}(t,x,D)+S^2 \one(t,x),\quad t>0, \ x\in D.
\end{equation}
Making use of the 
\kb{lower bound} in Hypothesis \ref{hyp1}, we can actually establish that $k$ is a Markovian kernel.

\begin{thm}\label{t.m}
Under Hypothesis~\ref{hyp1}, $\mk{k(t,x,D)}=1$ for all 
$t>0$, $x\in D$.
\end{thm}
\begin{proof}
In view of Lemma \ref{l.submarkovian}, it suffices to prove that $\mk{k(t,x,D)} \geq 1$ for all $t>0$ and $x\in D$. We fix an arbitrary $T\in (0,\infty)$ and proceed in two steps.\smallskip

\emph{Step 1:} We prove a positive lower bound for $\mk{p_D(t,x,D)} + S \mk{p_D}(t,x,D)$ uniform for $t\in (0,T)$ and $x\in D$. \mk{Let
$H$ and $\vartheta$ be as in Hypothesis \ref{hyp1}(ii)}.
By Lemma \ref{l.uniformstopping} and \eqref{e.tfsp2}, we find $\eta>0$ such that $\mk{p_D(s,w,D)} \geq \eta$ for
$s\in (0,T)$ and $w\in H$.
Then for $t\in (0,T)$ and $x\in D$, by \eqref{eq:a1} \kb{and \eqref{e.fsp}} we get
\begin{align}
S \mk{p_D}(t,x,D) &
 = \int_0^t\!\dx s \int_D\! \dx v \int_{D^c}\! \dz \int_D \, \mk{p_D(s,x,v)} \nu(v,z) \mu(z, \dx w) \mk{p_D(t-s,w,D)}\notag\\
& \ge \eta \int_0^t\dx s \int_D \dx v \,  \int_{D^c} \dz \, \mk{p_D(s, x,v)} \nu(v,z) \mu(z,\mk{H})\label{eq.Kest}\\
& \ge \mk{\eta\vartheta} \int_0^t\dx s \int_D \dx v \,  \int_{D^c} \dz\, \mk{p_D(s,x,v)} \nu(v,z)\notag \\
&= \mk{\eta\vartheta} \Pb^x ( \tau_D < t)\notag\\
&= \mk{\eta\vartheta}  \int_0^t\dx s \,  \int_D \dy \int_{D^c} \dz\int_D \,  \mk{p_D(s, x,y)} \nu(y,z) \mu(z,\dx w)\notag\\
& \ge  \mk{\eta\vartheta} \int_0^t\!\dx s  \int_D\! \dy \int_{D^c}\! \dz \int_D\,  
\mk{p_D(s, x,y)} \nu(y,z) \mu(z,\dx w)\mk{p_D(t-s, w,D)}\notag\\
&= \mk{\eta\vartheta} \mk{\kb{S} p_D}(t,x,D)\notag.
\end{align}
We conclude that all the above integrals are comparable. This will be quite useful later on, but for now we only deduce that
for $t\in (0,T)$ and $x\in D$,
\begin{align}\label{e.lbpSp}
\mk{p_D(t,x,D)} + S \mk{p_D}(t,x,D)
& \ge 
\Pb^x ( \tau_D > t)+ 
\mk{\eta\vartheta}\Pb^x ( \tau_D < t) \ge
\mk{\eta\vartheta}.
\end{align}
\smallskip

\emph{Step 2:} We prove that $\mk{k(t,x,D)}=1$ for $t>0$ and $x\in D$.
Indeed, let
\[
\ell = \inf\{ \mk{k(t,x,D)} : x\in D, t\leq T\}.
\]

Clearly $0\leq \ell \leq 1$. Iterating \eqref{e.pf}, for  $t\in (0,T)$ and $x\in D$ we obtain
\begin{align*}
\mk{k(t,x,D)}&= \mk{p_D(t,x,D)}+S \mk{p_D}(t,x,D)+S^2k(t,x,D)\\
& \geq \mk{p_D(t,x,D)}+S \mk{p_D}(t,x,D)+\ell S^2\one(t,x)\\
&= \ell\big[\mk{p_D}(t,x,D)+S \mk{p_D}(t,x,D)+S^2\one(t,x)\big]\\
&\quad +(1-\ell)\big[\mk{p_D}(t,x,D)+S \mk{p_D}(t,x,D)\big].
\end{align*}
By \eqref{e.11} and Step 1, 
$\ell\ge \ell + (1-\ell)\mk{\eta\vartheta}$, hence $\ell=1$, which ends the proof.
\end{proof}

\begin{cor}\label{c.ucSn}
We have $S^n\one(t,x) \le (1-\mk{\eta\vartheta})^{\lfloor n/2 \rfloor}$ for $n\in \N_0$, $t>0$, $x\in D$.
\end{cor}

\begin{proof}
From \eqref{e.11} and \eqref{e.lbpSp} we get $S^2\one\le (1-\mk{\eta\vartheta})\one$. Therefore,
$S^{2n}\one\le (1-\mk{\eta\vartheta})^n\one$, $n\in \N$. The statement follows from this and \eqref{e.1ip}.
\end{proof}
Theorem \ref{t.m} and Equation \eqref{e.dk} yield
\begin{equation}\label{e.18}
1=\mk{p_D(t, x,D})+S \mk{p_D}(t,x,D)+S^2 \mk{p_D}(t,x,D)+\ldots,\quad t>0, \ x\in D.
\end{equation}
Corollary~\ref{c.ucSn} shows that the series in \eqref{e.18} converges
exponentially.

\kb{We will focus on 
the \textit{semigroup}} $\mk{K=(K(t), t > 0})$ associated to the transition kernels $(\mk{k(t)}, t > 0)$. More precisely,
given $f\in B_b(D)$ we put 
\begin{equation}
\label{e.transition}
\mk{K(t)}f(x) \coloneqq \int_D f(y)\mk{k(t, x,\dx y)}, \quad t\kb{>} 0,\; x\in D.
\end{equation}
As a consequence of Theorem \ref{t.m}, $\mk{K(t)}$ is a Markovian operator on $B_b(D)$. 
It follows from Lemma \ref{l.ck} that the family $(\mk{K(t), t>0)}$ \kb{indeed} satisfies the semigroup law $\mk{K(t+s)=K(t)K(s)}$ for $\kb{s,t>0}$.

\mk{
\begin{rem}\label{r.rh}
    Many results of this section do not need the full strength of Hypothesis \ref{hyp1}. Namely, Lemma \ref{l.phi}, \ref{l.ck} and \ref{l.submarkovian} \kb{do not} require parts (ii) and (iii) of Hypothesis
    \ref{hyp1} and Theorem \ref{t.m} and Corollary \ref{c.ucSn}
    do not use part (iii).
    \kb{The Lipschitz condition on $D$ can be replaced throughout the paper by the regularity \eqref{e.reg} and ``not hitting the boundary" \eqref{e.nhtb}; this follows from our proofs. For instance, an open set $D\subset\Rd$ with the complement $D^c$ satisfying the so-called volume density condition \cite{BOGDAN2019} is regular by the arguments following \eqref{e.reg}.
If its boundary has zero Lebesgue measure, then \eqref{e.nhtb} holds too; see 
\cite[Corollary A.2]{BOGDAN2019}.
    The boundedness of $D$ may be dropped in Lemma \ref{l.phi}, \ref{l.ck}, \ref{l.submarkovian}, Theorem \ref{t.m}, and Corollary \ref{c.ucSn}, but 
it is
crucial, e.g., in Section~\ref{s.im}; see Remark~\ref{r.iob}.} 
\end{rem}
}

\kb{
\begin{example}\label{ex.dirac}
An important special case is when $\mu(z,A)$ does not depend on~$z$.
For instance, let $\mu(z,\cdot)=\delta_{x_0}$, the Dirac measure at a (fixed) point $x_0\in D$. 
By \eqref{eq.phi},
$\varphi(t,x,\dx w) = \left(P_D(t)\kappa_D\right)(x)\delta_{x_0}(\dx w)$,
where $\kappa_D(v)\coloneqq\nu(v,D^c)$ for $v\in D$. So, for $f\ge 0$, $t>0$, and $x,y\in D$,
\[
Sf(t,x,y) = \int_0^t \dx s P_D(s)\kappa_D(x)f(t-s,x_0,y)
=\left(P_D(\cdot)\kappa_D(x)*f(\cdot,x_0,y)\right)(t),
\]
where $*$ denotes the convolution on $\R$.
In particular, for $t>0$, $x,y\in D$,
\begin{equation}\label{e.SpD}
Sp_D(t,x,y)= \int_0^t \dx s \int_D \dx v \, p_D(s,x,v)\kappa_D(v)p_D(t-s,x_0,y),
\end{equation}
but it also follows that 
\begin{equation}\label{e.dkx0}
k(\cdot,x,y)=p_D(\cdot,x_0,y)*\sum_{n=0}^\infty \left(P_D(\cdot)\kappa_D(x)\right)^{*n},\quad \quad x,y\in D.
\end{equation}
\end{example}
\begin{rem}\label{r.tough}
The kernel $k$ from Example~\ref{ex.dirac} is quite singular.
Indeed, we have
$k(t,x,y)\ge Sp_D(t,x,y)$ and 
the inner (space) integral in \eqref{e.SpD} is
\begin{align*}
\int_D \dx v \, p_D(s,x,v)\kappa_D(v)
&\ge c\int_D \dx v \,p_D(s,x,v)
=c\Pb^x(\tau_D>s)\ge c \Pb^x(\tau_D>s_0),
\end{align*}
where $c\coloneqq \inf\{\nu(v,D^c): v\in D\}>0$, $s\le s_0<\infty$, and we used \eqref{e.tfsp2}.
It~follows that for $t\le s_0$ and $x$ in any given compact subset of $D$,
\begin{align*}
k(t,x,y)\ge C \int_0^t \dx s\, p_D(t-s,x_0,y)=C \int_0^t \dx s\, p_D(s,x_0,y),
\end{align*} 
\ml{where $C>0$ by Lemma \ref{l.uniformstopping}.}
Then, if $y$ is close to $x_0$,  $p_D(s,x_0,y)\approx p(s,x_0,y)$, so by \cite[p. 249]{MR2457489},
$$k(t,x,y)
\ge C'|y-x_0|^{\alpha-d}\wedge \left(t^2|y-x_0|^{-\alpha-d}\right).$$
Now, if $2\alpha<d$ and $f(y)\approx |y-x_0|^{-\alpha}$ on $D$, then we have $f\in L^2(D)$, but $\int_D k(t,x,y)f(y)\dx y\equiv \infty$. So the semigroup $(K(t))$ does not even act on $L^2(D)$, in particular it is not associated to a Dirichlet form. Moreover, the kernels $k(t,x,y)$ are not symmetric, otherwise symmetry, \ml{the} equality $K(t)\one_D=\one_D$, and Schur's test would make $K(t)$ a contraction on $L^2(D)$ for $t>0$, a contradiction. We will also see in Remark \ref{r.notfeller} that $(K(t))$ is not a Feller semigroup. This motivates our approach by resolvent kernels in the next section.
\end{rem}
}

\section{The Laplace transform of the semigroup}\label{s.Lt}

We now study the Laplace transform $\mk{R(\lambda)}$ of 
\kb{$(K(t))$}, defined by
\[
(\mk{R(\lambda)} f)(x)\coloneqq \mk{\int_0^\infty \dx t}\int_D e^{-\lambda t}\mk{k(t, x, \dx y)} f(y),\quad x\in D, \; \kb{\lambda>0,}
\]
and relate it to the Laplace transform $\mk{R_D(\lambda)}$ of $\kb{(P_D(t))}$.
To this end we introduce  the operator $\mk{\Phi(\lambda)}$, 
\[
(\mk{\Phi(\lambda)} f)(x) \coloneqq \int_0^\infty \dt  \int_D\, e^{-\lambda t} \varphi(t, x, \dx y) f(y),\quad \kb{x\in D,\; \lambda\ge 0,}
\]
where $f \in B_b(D)$.
\ml{This}   operator is closely related to the 
operator $\mk{H_D(\lambda)}$ \kb{from}
Section \ref{sect.dirichlet}. Indeed, 
\kb{since} $\mu$ is a kernel,  for $f\in B_b(D)$ we may define
\[
(\mu f)(z)\coloneqq \mu (z, f) \coloneqq \int_D \mu (z, \dx y) f(y), \quad z\in D^c.
\]
With this notation, we have $\mk{\Phi(\lambda)} f = 
\mk{H_D(\lambda)[\mu f]}$. 
From Lemma \ref{l.philambda_boundary-g} we now obtain the following result about continuity of $\mk{\Phi(\lambda)} f$. In the formulation of the result, we say that function $f\in C_b(D)$ \emph{belongs to $C(\overline{D})$} (and write $f\in C(\overline{D})$) if it has a (necessarily unique) continuous extension to $\overline{D}$. We then identify $f$ and its extension to $\overline{D}$.

\begin{lem}\label{l.philambda_boundary}
The operator $\mk{\Phi(\lambda)}$ has the strong Feller property on $D$.
If $f \in C_b(D)$ then $\mk{\Phi(\lambda)} f\in C(\overline D)$ and $\mk{\Phi(\lambda)} f=\mu f$ on $\partial D$.
\end{lem}

\begin{proof}
For $f\in B_b(D)$, $\mu f\in B_b(D^c)$ and Lemma~\ref{l.philambda_boundary-g} yields $\mk{\Phi(\lambda)} f\in C_b(D)$, which is the strong Feller property.
If $f\in C_b(D)$ then \mk{by Hypothesis \ref{hyp1}(iii)\kl{,} the function
$\mu f$ is continuous at $\partial D$ and} 
Lemma \ref{l.philambda_boundary-g} \mk{yields that} $\mk{\Phi(\lambda)} f\in C(\overline D)$ and
$\mk{\Phi(\lambda)} f=\mu f$ at $\partial D$.
\end{proof}

We note that for $f \in B_b(D)$,
\begin{equation}\label{e.fPhlp}
\mk{\Phi(\lambda)} f(x)=
\expect^x \big[ e^{-\lambda \tau_D}\mu f(Y_{\tau_D}) \big], \quad x\in \overline{D}.
\end{equation}

\begin{lem}\label{l.i}
If $f\in B_b(D)$, $\lambda>0$ and $\mk{\Phi(\lambda)} 
f=f$, then $f=0$.
\end{lem}

\begin{proof}
By Lemma~\ref{l.philambda_boundary}, $f=\mk{\Phi(\lambda)} f\in C(\overline{D})$. 
Assume that $\sup_{\overline{D}} f>0$. Note that $\sup_{D^c} \mu f\le \sup_D f$. Using \eqref{e.fPhlp}, for every $x\in D$ we get
\[
\mk{\Phi(\lambda)} f(x)\le 
\expect^x \big[ e^{-\lambda \tau_D}\big]\sup_{\overline{D}} f<\sup_{\overline{D}} f,
\]
because  $\mpr^x$-$a.s.$ we have $\tau_D>0$, by the right-continuity of the trajectories of the process $Y$.
In particular, the maximum of $f$ is attained at $\partial D$. 
\mk{Let $H$ and $\vartheta$ be as in Hypothesis \ref{hyp1}(ii).
Then} $\sup_{H} f =
(1-\delta)\sup_D f$ for some $\delta>0$. 
\mk{As $\Phi(\lambda)f = f$, Lemma \ref{l.philambda_boundary} yields for $z\in \partial D$}, 
\begin{align*}
f(z)&=\mu f(z)=\int_D f(x)\mu(z,\dx x) \leq (1-\delta)\mu(z,H)\sup_{\overline{D}} f + \mu(z,D\setminus H)\sup_{\overline{D}} f  \\
&= \sup_{\overline{D}} f-\delta \mu(z,\mk{H})\sup_{\overline{D}} f\le (1-\delta\mk{\vartheta})\sup_{\overline{D}}f<\sup_{\overline{D}}f,
\end{align*}
a contradiction. So, $\sup_{\overline{D}} f\le 0$. By linearity, $\sup_{\overline{D}} (-f)\le 0$\kb{;}  $f=0$ on $D$.
\end{proof}

\begin{lem}\label{l.phiseries}
We consider $\mk{\Phi(\lambda)}$ as an operator on $C(\overline{D})$. Then the series 
\[
\sum_{n=0}^\infty \mk{\Phi(\lambda)^n}
\]
converges in operator norm for $\lambda >0$.
\end{lem}

\begin{proof}
It follows from Lemma \ref{l.philambda_boundary}, that $\mk{\Phi(\lambda)^2}$ defines a strong Feller operator on $\overline{D}$. As is \mk{well known}, see \cite[\S 1.3]{revuz}, its square, i.e., $\mk{\Phi(\lambda)^4}$, is an ultra-Feller operator, i.e., it maps bounded subsets of $B_b(\overline{D})$ to equicontinuous subsets of $C(\overline{D})$. In particular, $\mk{\Phi(\lambda)}^4$ is a compact operator. By a variant of the Fredholm alternative, see \cite[Theorem 15.4]{krein82}, $I-\mk{\Phi(\lambda)}$ is invertible if and only if it is injective. The latter was proved in Lemma \ref{l.i}. So, $1$ belongs to the resolvent set of $\mk{\Phi(\lambda)}$. By the Krein--\mk{Rutman} Theorem, see \cite[Proposition V.4.1]{schaefer}, the spectral radius $r(\mk{\Phi(\lambda)})$ belongs to the spectrum of $\mk{\Phi(\lambda)}$. Since $\|\mk{\Phi(\lambda)}\| \leq 1$ and $1$ belongs to the resolvent set of $\mk{\Phi(\lambda)}$, we must have $r(\mk{\Phi(\lambda)})< 1$, which is equivalent to the claim.
\end{proof}

We can now relate the resolvents \mk{$R(\lambda)$ 
and 
$R_D(\lambda)$.}
\begin{lem}\label{l.resolvents}
For $\lambda>0$ and $f \in B_b(D)$ we have $\mk{R(\lambda)} f = \sum_{n=0}^\infty \mk{\Phi(\lambda)^n R_D(\lambda)} f$.
In particular,  the identity $\mk{R(\lambda) = R_D(\lambda) +\Phi(\lambda) R(\lambda)}$ holds true.
\end{lem}

\begin{proof}
To prove the lemma, we make use of the series representation \eqref{e.dk} for the kernel $k$. Let us first see how $\mk{\Phi(\lambda)}$ interacts with the operator $S$. To that end, let $h : (0,\infty)\times D \to [0,\infty)$ and $x\in D$.
By Tonelli's theorem, 
\begin{align*}
& \int_0^\infty\dt  \, e^{-\lambda t}S h(t,x)\\
=& \quad \int_0^\infty\dt  \, e^{-\lambda t}\int_0^t \dx s\int_D \varphi (s, x, \dx w)  h(t-s, w)\\
= & \int_0^\infty \dx s \int_s^\infty \dt  \, e^{-\lambda t}\int_D \varphi (s, x, \dx w)  h(t-s, w)\\
= & \int_0^\infty \dx s \int_{\mk{D}}  e^{-\lambda s}\varphi (s, x, \dx w) \int_0^\infty \, dr \, e^{-\lambda r}  h(r, w)\\
= &  \Big(\mk{\Phi(\lambda)} \int_0^\infty e^{-\lambda r } h(r, \cdot)\, \dx r\Big)(x).
\end{align*}
Summarizing, we obtain the Laplace transform (in $\lambda$) of $Sh$ by applying $\mk{\Phi(\lambda)}$ to the Laplace transform
of $h$.
By this observation and induction, \eqref{e.dk} yields
\[
\mk{R(\lambda)} f = \sum_{k=0}^\infty \mk{\Phi(\lambda)^kR_D(\lambda)} f,
\]
as claimed.
\end{proof}
\kb{
\begin{example}\label{e.Phil}
For $\mu$ in Example~\ref{ex.dirac}, 
 $(\Phi(\lambda)f)(x)=f(x_0) \left(R_D(\lambda)\kappa_D\right)(x)$, so 
\begin{align*}
R(\lambda)f(x)&=R_D(\lambda)f(x)+R_D(\lambda)\kappa_D(x)\, R_D(\lambda)f(x_0)\sum_{n=0}^\infty R_D(\lambda)\kappa_D(x_0)^n \\
&=R_D(\lambda)f(x)+R_D(\lambda)\kappa_D(x)\,R_D(\lambda)f(x_0)/\left(1-R_D(\lambda)\kappa_D(x_0)\right),
\end{align*}
for $x\in D$, $\lambda>0$, and $f \in B_b(D)$.
\end{example}
}
We now come to the main result of this section, in which we characterize the closure of the range of $\mk{R(\lambda)}$. Given a function
$f\in C_b(D)$, we let
\begin{equation}\label{eq.fmu}
f_\mu (x) \coloneqq \begin{cases} f(x), & \mbox{ for } x\in D,\\
\mu (x, f), & \mbox{ for } x\in D^c,
\end{cases}
\end{equation}
and we define the space $C_\mu (D)$ by
\begin{equation}\label{eq.cmu}
C_\mu (D) \coloneqq \{ f\in C_b(D) : f_\mu \mk{\mbox{ is continuous on } \overline{D}}\}.
\end{equation}
\mk{\kb{By Hypothesis \ref{hyp1}(iii),} the map $D^c\ni x\mapsto \mu(x, f)$ is 
continuous 
on $\partial D$. Thus, the condition that $f_\mu$ is continuous on $\overline{D}$ is equivalent
with $f(x_n) \to \mu(x, f)$ whenever $(x_n)\subset D$ converges
to $x\in \partial D$. We 
can \kb{therefore}, similar to the remarks
preceding Lemma \ref{l.philambda_boundary}, identify 
$C_\mu(D)$
with the following (closed) subspace of $C_b(\overline{D})$:
\begin{equation}\label{eq.cmu2}
C_\mu (\overline{D}) \coloneqq \Big\{ f\in C_b(\overline{D}) : 
f(z) = \int_D f(x) \mu(z, \dx x) \mbox{ for all } z\in \partial D\Big\}.
\end{equation}
}

\mk{
\begin{example}\label{ex.1}
    Let $D= B(0,1)\subset\Rd$
be the ball of radius 1 centered at $0$ 
and $\mu(z, \cdot) = \delta_{0}$
    for $z\in D^c$; see Example \ref{ex.dirac}. In
    this case,
    \[
    C_\mu(\overline{D}) = \{ f\in C_b(\overline{D}) : 
    f(z) = f(0) \mbox{ for all } z\in \partial D\}
    \]
    and the extension $f_\mu$ is given by
    $f_\mu(x) = f(0)$ for $x\in D^c$.
\end{example}
}

\mk{\begin{example}\label{ex.2}
    Let $D= B(0,1)$ be as in Example~\ref{ex.1}, but let
    \[
    \mu(z, \cdot) = \Big(\frac{1}{2} + \frac{1}{2|z|}\Big)
    \delta_0 + \Big(\frac{1}{2} - \frac{1}{2|z|}\Big)\delta_{(1-|z|^{-1})\frac{z}{|z|}},\qquad z\in D^c.
    \]
    Then $\mu$ satisfies Hypothesis \ref{hyp1}. Note that
    $\mu(z, \cdot) = \delta_0$ when $|z|=1$\kb{,} so the space
    $C_\mu(\overline{D})$ is the same as in Example \ref{ex.1}, but the extension $f_\mu$ is different: 
    \[
    f_\mu(x) = \Big(\frac{1}{2} + \frac{1}{2|x|}\Big)
    f(0) + \Big(\frac{1}{2} - \frac{1}{2|x|}\Big)f\Big((1-|x|^{-1})\frac{x}{|x|}\Big),\quad x\in D^c.   \]
\end{example}}

\mk{
\begin{example}\label{ex.3}
    Let $\mu(z, \cdot) = \lambda_d(D)^{-1}\lambda_d$, 
    where $\lambda_d$ is the
$d$-dimensional Lebesgue measure on $D$. 
    Denote
    $\bar{f}\coloneqq \lambda(D)^{-1}\int_D f(x) \dx x$. Then $f_\mu(x) = \bar{f}$, $x\in D^c$, and
    \[
    C_\mu(\overline{D}) = \{ f\in C_b(\overline{D}) : 
    f(z) = \bar{f} \mbox{ for all } z\in \partial D\}.
    \]
    \end{example}
}

\mk{Note that we may} rephrase \kb{the second statement of} Lemma \ref{l.philambda_boundary} by saying that $\mk{\Phi(\lambda)} f \in C_\mu (D)$ for all $f\in C_b(D)$.

\begin{thm}\label{t.closurerange}
For $\lambda>0$, the closure of the range of $\mk{R(\lambda)}$ equals $C_\mu(D)$.
\end{thm}

\begin{proof}
Let us first prove that the range of $\mk{R(\lambda)}$ is contained in $C_\mu(D)$. 
To that end, fix $f\in B_b(D)$. As $\mk{P_D(t)}f \in C_0(D)$ we have $\mk{R_D(\lambda)} f \in C_0(D)\subset C(\overline{D})$ for any $f\in B_b(D)$. Using Lemma \ref{l.philambda_boundary} and induction, we have $\mk{\Phi(\lambda)^kR_D(\lambda)} f \in C_\mu (D) \subset C(\overline{D})$ for all $k \geq 1$ and  Lemma \ref{l.phiseries} and Lemma \ref{l.resolvents} imply that $\mk{R(\lambda)} f \in C(\overline{D})$.

Now fix $x_0 \in \partial D$. Putting $u = \mk{R(\lambda)}  f$, we find
\begin{align*}
u(x_0) & = \mk{R(\lambda)}  f(x_0)  = \mk{R_D(\lambda)}  f(x_0)  + \sum_{k=1}^\infty \Big(\mk{\Phi(\lambda)^k R_D(\lambda)} f\Big)(x_0)\\
& = 0 + \sum_{k=1}^\infty \mu \big( x_0, \mk{\Phi(\lambda)^{k-1} R_D(\lambda)}  f\big)\\
& = \mu \Big( x_0 , \sum_{k=1}^\infty \mk{\Phi(\lambda)^{k-1} R_D(\lambda)} f\Big) = \mu (x_0, u).
\end{align*}
Here the second equality uses Lemma \ref{l.resolvents}, 
the third Lemma \ref{l.philambda_boundary} and the fact that $\mk{R_D(\lambda)} f \in C_0(D)$. The fourth equality uses dominated convergence and the last Lemma \ref{l.resolvents} again. This shows that $u\in C_\mu(D)$. As the latter is closed and contains the range of $\mk{R(\lambda)}$, it also contains the closure of the range.\smallskip

To prove the converse, we only need to show that the range of $\mk{R(\lambda)}$ is dense in $C_\mu(D)$. To that end, let $f\in C_\mu(D)$ and $g\coloneqq f-\mk{\Phi(\lambda)} f$. By Lemma \ref{l.philambda_boundary}, $g\in C_0(D)$. Since the semigroup of the killed process is strongly continuous on $C_0(D)$, the domain of its generator is 
dense in $C_0(D)$. We thus find a bounded sequence $(u_n)\subset C_0(D)$ such that $\mk{R_D(\lambda)} u_n \to g$ with respect to $\|\cdot\|_\infty$.

Next observe that
\begin{align*}
\sum_{k=0}^N \mk{\Phi(\lambda)^k R_D(\lambda)} u_n & \to \sum_{k=0}^N \mk{\Phi(\lambda)}^k g\\
&= \sum_{k=0}^N \big(\mk{\Phi(\lambda)}^k f - \mk{\Phi(\lambda)}^{k+1}f\kb{\big)} = f- \mk{\Phi(\lambda)}^{N+1}f
\end{align*}
as $n\to \infty$. Given $\eps>0$, we may, as a consequence of Lemma \ref{l.phiseries}, pick $N$ so large, that $\sum_{k\geq N} \|\mk{\Phi(\lambda)}^k\| \leq \eps$. Taking Lemma \ref{l.resolvents} into account, we get
\begin{align*}
\limsup_{n\to\infty} \|\mk{R(\lambda)} u_n - f\|_\infty &\leq C\eps\|\mk{R_D(\lambda)} \|_\infty + \|\mk{\Phi(\lambda)}^{N+1}f\|_\infty\\
&\leq C\|\mk{R_D(\lambda)}\|_\infty\eps +\eps \|f\|_\infty.
\end{align*}
As $\eps>0$ was arbitrary, we see that $\mk{R(\lambda)} u_n \to f$.
\end{proof}

\section{The \kb{generator of the}
semigroup}\label{s.sg}

\begin{prop}\label{p.strongcontinuity}
$\mk{K(t)}f \to f$ uniformly
as $t\to 0$ if, and only if, $f\in C_\mu(D)$.
\end{prop}

\begin{proof}
As is well known (see, e.g.\ \cite[Remark \ 2.5]{kunze09}), $\mk{K(t)}f \to f$ in the norm $\|\cdot\|_\infty$ if $f$ belongs to the domain of the generator of $\mk{K}$, i.e., the range of $\mk{R(\lambda)}$. By Theorem \ref{t.closurerange}, the latter is a dense subset of $C_\mu(D)$ so a standard approximation argument shows that the same is true for every $f\in C_\mu(D)$.

To see the converse, let $X\coloneqq \{f\in B_b(D) : \mk{K(t)}f \to f \mbox{ as } t\to 0\}$. Then $X$ is a closed subspace of $B_b(D)$ that is invariant under the semigroup $\mk{K}$. Moreover, the restriction of $(\mk{K(t)})$ to $X$ is strongly continuous. By the Hille--Yosida theorem, the generator of the restriction to $X$, which is nothing more than the part of the full generator in $X$, is dense in $X$. But then $X$ must be contained in the closure of the range of $\mk{R(\lambda)}$, i.e., $C_\mu(D)$.
\end{proof}

\begin{rem}\label{r.notfeller}
It follows from Proposition \ref{p.strongcontinuity} that the semigroup $\mk{K}$ is \emph{not}, in general, a Feller semigroup, since typically
$C_0(D) \not\subset C_\mu (D)$, whence the orbits of functions in $C_0(D)$ are not $\|\cdot\|_\infty$-continuous.
\end{rem}

We can now prove the first main result of this section.

\begin{thm}\label{t.strongfeller}
$\mk{K}$ 
is a $C_b$-semigroup and has the strong Feller property.
\end{thm}

\begin{proof}
Let $f\in B_b(D)$ and $x\in D$. From Equation \eqref{e.pf}, we obtain
\begin{equation}\label{e.perturb}
(\mk{K}(t)f)(x) = (\mk{P_D(t)}f)(x) + \int_D Sk(t,x,y) f(y)\dx y.
\end{equation}
Let us consider the second term 
on the right hand side of \eqref{e.perturb}. 
We have
\begin{align}
\Big|\int_D Sk(t,x,y) f(y)\, dy \Big| & \leq \int_0^t \dx s \int_D \dx w \, \varphi (s,x,w) \mk{k(t-s, w, D)} \|f\|_\infty\notag\\
& = \|f\|_\infty \int_0^t \dx s \int_D\dx w \, \varphi (s,x,w) = \|f\|_\infty \mpr^x(\tau_D \leq t),\label{e.stopest}
\end{align}
where the last equality uses Lemma \ref{l.phi}(b) and Equation \eqref{e.tfsp2}. By \cite[Lemma 2]{chung86}, the latter converges to $0$ as $t\to 0$, uniformly on compact subsets of $D$.

Let now $f\in C_b(D)$. We have seen that the integral in \eqref{e.perturb} converges locally uniformly to $0$ as $t\to 0$. By Lemma \ref{l.pd}, $\mk{P_D(t)}f$ converges locally uniformly to $f$. Thus, \eqref{e.perturb} yields that $\mk{K(t)}f \to f$ locally uniformly as $t\to 0$.\smallskip

Let us now prove the strong Feller property. To that end, fix $t>0$, $x\in D$ and $f\in B_b(D)$. Note that for
$s \in (0,t)$ we have $\mk{K(t)}f = \mk{K(s)K(t-s)}f$. We set $g_s \coloneqq \mk{K(t-s)}f$. By \eqref{e.perturb} (with $f$ replaced by $g_s$ and $t$ replaced by $s$) and \eqref{e.stopest},
\[
|\mk{K(t)}f(x) -\mk{P_D(s)} g_s(x)| \leq \|g_s\|_\infty \mpr^x(\tau_D\leq s) \leq \|f\|_\infty \mpr^x(\tau_D \leq s).
\]
The latter converges locally uniformly to $0$ as $s\to 0$, so $\mk{P_D(s)} g_s$ converges locally uniformly to $\mk{K(t)} f$. Since $\mk{P_D}$ has the strong Feller property (Lemma \ref{l.pd}), the functions $\mk{P_D(s)}g_s$ are continuous. But then so is $\mk{K(t)}f$.
\end{proof}

We can now characterize the $C_b$-generator $\cA$ of the semigroup $\mk{K}$. In the following theorem, we use the notation $u_\mu$ introduced in \eqref{eq.fmu}.

\begin{thm}\label{t.generator}
For $u, f \in C_b(D)$, the following are equivalent:
\begin{enumerate}
[{\rm (i)}]
\item $u\in D(\cA)$ and $\cA u = f$.
\item 
$u\in C_\mu (D)$ and
\[
f(x) = \lim_{t\to 0} \frac{\mk{P(t)} u_\mu (x) - u_\mu (x)}{t}, \qquad x\in D.
\]
\item 
$u\in C_\mu (D)$ and
\[
f(x) = \lim_{\eps \to 0^+} \int_{\{|y-x|>\eps\}} \big[ u_\mu (y) - u_\mu (x)\big]\nu (x,y)\dx y, \qquad x\in D.
\]
\item 
$u\in C_\mu (D)$ and, with $\gamma$ given by \eqref{e.kgKt},
\[
f(x) = \lim_{\eps \to 0^+} \int_{\{|y-x|>\eps\}\cap D} (u(y)-u(x))\gamma(x,\dx y), \qquad x\in D.
\]
\end{enumerate}
\end{thm}

\begin{proof}
By Lemma \ref{l.resolvents}, (i) is equivalent with $u = \mk{[R(\lambda)]}(\lambda u - f) = \mk{R_D(\lambda)}(\lambda u - f) +\mk{\Phi(\lambda)} u$. 
Thus (i) is equivalent to $u-\mk{\Phi(\lambda)} u =
\mk{R_D(\lambda)}(\lambda u - f)$, which, in turn, is equivalent to
$u-\mk{\Phi(\lambda)} u \in D(\cA^D)$ and $\cA^D (u-\mk{\Phi(\lambda)} u ) = 
f - \lambda \mk{\Phi(\lambda)} u$.
By Theorem \ref{t.closurerange}, $D(\cA) \subset C_\mu (D)$, so the
equivalence of (i) and (ii) follows from Proposition \ref{p.killedgenerator} and Lemma \ref{l.harmonic}, 
applied with $g= \mu f$.

To prove the implication (i) $\Rightarrow$ (iii), we note that, by Proposition \ref{p.killedgenerator} and (the proof of) Lemma \ref{l.harmonic},
the function $\mk{R_D(\lambda)}(\lambda u - f) +\mk{\Phi(\lambda)} u$ belongs to the domain of the Dynkin operator and $\mathscr{D} \big[\mk{R_D(\lambda)} (\lambda u - f) +\mk{\Phi(\lambda)} u\big] 
=f-\lambda \mk{\Phi(\lambda)} u
+ \lambda \mk{\Phi(\lambda)} u = f$ on all of $D$. At this point,
\cite[Lemma 3.3]{MR3613319} yields (iii). 

The  implication (iii) $\Rightarrow$ (ii) follows once again from \cite[Lemma 3.4]{MR3613319}, whereas (iv) is merely a reformulation of (iii).
\end{proof}

\begin{rem}\label{r.bc}
The condition $u\in C_\mu (D)$ that appears in the statements (ii) -- (iv) of Theorem \ref{t.generator} can be seen as a \emph{boundary condition} that a function $u\in C_b(D)$ necessarily satisfies if it belongs to $D(\cA)$. The condition is equivalent to
\[
\lim_{D\ni x\to z} u(x) = \int_D u(y) \mu (z, \dx y), \qquad z\in \partial D.
\]
\end{rem}

\begin{cor}\label{c.se}
Let $\gamma$ be given by \eqref{e.kgKt} and fix $\lambda >0$. Then for every $f\in C_b(D)$ there exists a unique function 
$u\in C_\mu(D)$ satisfying
\[
\lambda u(x) - \lim_{\eps\to 0^+}\int_{\{|y-x|>\eps\}\cap D} (u(y) - u(x))\gamma(x, \dx y) = f(x),\quad x\in D.
\]
\end{cor}

\begin{proof}
Since $\cA$ is the generator of a $C_b$-Feller semigroup, $(0,\infty)$ belongs to the resolvent set of $\cA$. Thus, for every $\lambda >0$ and
$f\in C_b(D)$ the \mk{e}quation $\lambda u - \cA u = f$ has a unique solution $u\kb{\coloneqq R(\lambda)}\in D(\cA)$. Now the claim follows from the characterization of $\cA$ in Theorem \ref{t.generator}(iv).
\end{proof}

\section{
Asymptotic behavior \kb{of the semigroup}}\label{s.im}

In order to establish the existence of an invariant measure, we employ the lower-bound technique of Lasota  \cite[Theorem~6.1]{MR1452617}. Here is the first step.

\begin{lem}\label{l.lowerbound}
Fix $t>0$ \mk{and let $H$ be as in Hypothesis \ref{hyp1}(ii)}.
\begin{enumerate}
[(a)]
\item \mk{There is $\delta >0$} such that $\mk{k(t,x,y)} \geq \delta$ for all $x\in D$, $y\in H$.
\item \mk{There is $\delta >0$} such that $\mk{k(s,x,H)} \geq \delta\mk{|H|}$ for all $x\in D$ and $s\geq t$.
\item There is $\eps>0$ such that $\int\limits_D \!|\mk{k(t, x_1,y) \!-\! k(t, x_2,y)}|\dy\! \leq \!2-\eps$ for all $x_1,x_2\in D$.
\end{enumerate}
\end{lem}
\begin{proof}
(a)  Since $\mk{p_D}$ is continuous and positive, by compactness there is 
a constant $c=c(D,H,t,\alpha)>0$ such that $\mk{p_D}(r, w, y) \geq c$ for all $r\in [t/2, t]$ and $w,y\in H$.
Then, for $x \in D$, $y\in H$, we get
\begin{align}
\mk{k(t,x,y)}&\ge S \mk{p_D}(t,x,\mk{y})\notag\\
&\ge \int_0^{t/2}\dx s \int_D \dx v \int_{D^c} \dz \int_H \, p^D_s(x,v) \nu(v,z) \mu(z, \dx w) \mk{p_D(t-s,w,y)}\notag\\
&\ge \mk{c\vartheta}\int_0^{t/2}\dx s \int_D \dx v \int_{D^c} \dz \, \mk{p_D(s,x,v)} \nu(v,z)= \mk{c\vartheta}\Pb^x ( \tau_D < t/2)\label{eq.stopprob}.
\end{align}
Note that $x\mapsto \mk{p_D(s,x,v)}\nu(v,z)$ is strictly positive and continuous for almost all triplets $(s, v, z)$. Fatou's lemma implies that
the function $x\mapsto \Pb^x(\tau_D < t/2)$ is lower semicontinous. \kb{Because of the regularity \eqref{e.reg}, the}
function tends to $1$ as $x$ approaches the boundary (see \cite[Theorem 1.23]{MR1329992}), so at some point of $D$ it attains its minimum, say $C>0$.
\mk{Thus} $\mk{k(t, x,y)}\geq cC\mk{\vartheta}=:\delta$ for all $x\in D$, $y\in H$.

(b) 
By the Chapman--Kolmogorov equations in Lemma \ref{l.ck}, for $s> t$ we get
\[
\mk{k(s,x,H)} = \int_D \mk{k(s-t, x,y) k(t, x,H)}\dy \geq \delta\mk{|H|}.
\]

(c) Pick again $\delta$ as in (a). By making $H$ larger, we may assume that $|H|>0$. Then, for all $x_1,x_2\in D$ we have
\begin{align*}
&\int_D|\mk{k(t, x_1,y)-k(t, x_2,y)}|\dy\\&=\int_{D\setminus H}|\mk{k(t, x_1,y)-k(t, x_2,y)}|\dy
+\int_{H}|\mk{k(t,x_1,y)}-\delta-(\mk{k(t,x_2,y)}-\delta)|\dy\\
&\le \int_{D\setminus H}\mk{k(t, x_1,y)}\dy+\int_{D\setminus H}\mk{k(t,x_2,y})\dy\\
&\quad +\int_{H}(\mk{k(t,x_1,y)}-\delta)\dy+\int_{H}(\mk{k(t,x_1,y)}-\delta)\dy=2-2\delta|H|,
\end{align*}
so we can take $\eps=2\delta|H|>0$.
\end{proof}

We next consider the adjoint of the operators $\mk{K(t)}$, $t>0$. For a finite measure $\kappa$ on $D$ we put
\[
\mk{K(t)}^* \kappa(A)\coloneqq \int_D \mk{k(t, x,A)}\kappa(\dx x),
\qquad t\kb{>} 0, \quad A\subset D.
\]
Then $\mk{K(t)}^*\kappa$ is again a measure on $D$ and, by Tonelli's theorem and Theorem \ref{t.m},
$\mk{K(t)}^* \kappa(D)=\kappa(D)$. Moreover, 
$\mk{K(t)}^*\kappa$ is absolutely continuous
with respect to Lebesgue measure, 
so we may think of $\mk{K(t)}^*$ as operating on $L^1(D)$. 
With this in mind, the operators $\mk{K(t)}^*$ are Markov operators
in the sense of Lasota and Mackey \cite{MR1244104}, 
Lasota and York \cite{MR1265226} and Komorowski \cite{MR1162571}. 
We call a probability measure $\kappa$ a \emph{stationary distribution} if $\mk{K(t)}^*\kappa = \kappa$ for all $t>0$. 
Then, again,
$\mk{K(t)}^*\kappa$ is absolutely continuous with respect to the Lebesgue measure so that any stationary distribution must have a density, called
\emph{stationary density}.

\begin{thm}\label{t.inv}
There is a unique stationary distribution $\kappa$. Moreover, there exist $M, \omega\in (0,\infty)$ such that for every probability measure
$\nu$ on $D$,
\[
\|\mk{K(t)}^*\nu - \kappa \|_\tv \leq Me^{-\omega t}\kb{,\quad t>0}.
\]
\end{thm}

\begin{proof}
By Lemma~\ref{l.lowerbound}(b), $\liminf_{T\to \infty} \frac1T\int_0^T \mk{K(t,x,H)}\dt >0$ for $x\in D$. This is the lower bound mentioned at the beginning of the section. By Da Prato and Zabczyk \cite[Remark 3.1.3]{MR1417491}, 
we get a stationary distribution.
Since the proof in not given in \cite[Remark 3.1.3]{MR1417491} we refer the reader to
Lasota and York \cite[Theorem 3.1]{MR1265226}, Lasota \cite[Theorem 6.1]{MR1452617} or 
Komorowski \cite[Theorem 3.1]{MR1162571} for the proofs in the discrete-time case.\smallskip

In view of Lemma \ref{l.lowerbound}(c), \cite[Theorem 1.3]{kulik15} (see also \cite{hairer21}), applied to $P=\mk{K(1)}$ yields the 
existence of constants $c>0$ and $e^{-\omega}=\gamma>0$ 
such that $\|\mk{K(n)}^*\nu - \kappa\|_\tv \leq c\gamma^n$ for all 
$n\in \N$ and probability measures $\nu$ on $D$. 
An arbitrary number $t>0$ can be written as $t=n+r$, where $n\in \N_0$ and $r\in [0,1)$. Then\kb{,}
\begin{align*}
\|\mk{K(t)}^*\nu - \kappa\|_\tv & = \|\mk{K(r)^*K(n)}^*\nu - \mk{K(r)}^*\kappa\|_\tv \leq \|\mk{K(n)}^*\nu - \mk{K(r)}^*\kappa\|_\tv\\
& \leq c e^{-\omega n} \leq ce^{\omega} e^{-\omega t} =: M e^{\omega t}.\qedhere
\end{align*}
\end{proof}

\kb{
\begin{rem}\label{r.iob}
Here are some observations and open problems.
\begin{enumerate}
[(i)]
\item The compactness of $\overline{D}$ is crucial in this section; 
Lemma~\ref{l.lowerbound} and Theorem~\ref{t.inv} easily fail for $D=\Rd$, since then $K(t)=P(t)$. On the other hand, we conjecture that they hold if the boundedness of $D$ is replaced by that of $x\mapsto \expect_x \tau_D$, e.g., 
if $D$ is an (unbounded) right circular cylinder.
\item The repulsion kernel $\mu$ enters the construction of $k(t)$ only through the \textit{return kernel} $V(x,A)\coloneqq \int_{D^c}\nu(x,z)\mu(z,A)\dx z$. We expect results for kernels  $V$ on $D$ satisfying $V(\cdot,D)\le \nu(\cdot,D^c)$ describing general reflections.
\item It is interesting to apply our methods to study semigroups of resurrections or recurrent extensions of Markov processes with scaling, see 
\cite{MR2266714, MR2364226, MR4140082, MR4514832, MR4520527}.
In the same vein, it is important to study the semigroup 
corresponding to the quadratic form and the Neumann problem in 
\cite{MR3651008};  see also \cite{BOGDAN2019, MR4245573} for motivation. In these settings, Hypothesis~\ref{hyp1}(ii) fails.
\end{enumerate}
\end{rem}
}

\appendix

\section{\texorpdfstring{$C_b$-Feller semigroups}{}}\label{a.sg}

By Remark \ref{r.notfeller}, the semigroup $(\mk{K(t), t  > 0})$ is, in general, \emph{not} a Feller semigroup, so in this paper we use a different semigroup concept, namely the notion of $C_b$-Feller semigroup. This can be seen as a special case of the theory of ``semigroups on norming dual pairs'', introduced in \cite{kunze09, kunze11}. As this is not a standard notion, we introduce this concept in this appendix and reformulate the relevant results from \cite{kunze09, kunze11} in our special case.\smallskip

Throughout, we let $E\subset$ be an open \mk{ subset of $\Rd$}, or, more generally, a Polish space. 
A \emph{kernel} on $E$ is a map $k: E\times \cB (E)\to \C$ such that (i) the map $x \mapsto k(x,A)$ is measurable for all $A\in \cB(E)$ (ii) the map $A\mapsto k(x,A)$ defines a measure on $E$ for every $x\in E$ and (iii) we have $\sup_x |k|(x, E) <\infty$, where $|k|(x,\cdot)$ refers to the total variation of $k(x,\cdot)$.

A bounded linear operator $T$ on $C_b(E)$ is called a \emph{kernel operator}, if there exists a kernel $k$ such that
\begin{equation}
\label{eq.kernelrep}
(Tf)(x) = \int_E f(y)\, k(x, \dx y) \qquad f\in C_b(E), \, x\in E.
\end{equation}

As it turns out, being a kernel operator can be characterized by an additional continuity condition with respect to the weak topology 
$\sigma \coloneqq \sigma (C_b(E), \cM(E))$ induced by the space of bounded (complex/signed) measures. We note that for a sequence of functions $(f_n) \subset C_b(E)$, the convergence with respect to $\sigma$ is nothing else than \emph{bp-convergence} (bp is short for bounded, pointwise), which means $\sup_n \|f_n\|_\infty <\infty$ and $f_n\to f$ pointwise. Indeed, that bp-convergence implies $\sigma$-convergence follows from the dominated convergence theorem whereas the converse implication follows easily using the uniform boundedness principle.

\begin{lem}\label{l.kernelop}
Let $T\in \cL (C_b(E))$ be a bounded linear operator. The following are equivalent:
\begin{enumerate}
[(i)]
\item $T$ is a kernel operator;
\item $T$ is $\sigma$-continuous;
\item If $(f_n) \subset C_b(E)$ bp-converges to  $f\in C_b(E)$, then $Tf_n$ bp-converges to $Tf$.
\end{enumerate}
\end{lem}	

\begin{proof}
The equivalence of (i) and (ii) is proved in \cite[Proposition 3.5]{kunze11} and the implication (ii) $\Rightarrow$ (iii) is trivial in view of the above comment. To see (iii) $\Rightarrow$ (i), let $\varphi (f) = (Tf)(x)$. By (iii), it follows that $\varphi (f_n) \to 0$ whenever $f_n$ bp-converges to $0$. Now \cite[Theorem 7.10.1]{bogachev} implies that $\varphi (f) = \int_E f\, d\nu_x$ for some Baire (hence Borel, as $E$ is Polish) measure $\nu_x$. The measurable dependence of $\nu_x$ on $x$ can now be proved in a standard way, see the proof the implication (i) $\Rightarrow$ (ii) in \cite[Proposition 3.5]{kunze11}.
\end{proof}

In what follows, the space of bounded, $\sigma$-continuous operators (equivalently: kernel operators) on $C_b(E)$ is denoted by $\cL(C_b(E), \sigma)$. Note that any operator $T\in \cL(C_b(E), \sigma)$ can uniquely be extended to a bounded linear operator on all of $B_b(E)$, by merely plugging $f\in B_b(E)$ into the right hand side of \eqref{eq.kernelrep}. In what follows we do not distinguish between $T$ and its extension to $B_b(E)$.

We are now ready to define what a $C_b$-semigroup is. To simplify the exposition, we restrict ourselves to sub-Markovian semigroups, as all the semigroups appearing in this article have this property. Obviously, a kernel operator $T$ with kernel $k$ is (sub-)Markovian if and only if the kernel $k$ is (sub-)Markovian, i.e.\ $k(x, \cdot)$ is a (sub-)probability measure for every $x\in E$.

\begin{defn}\label{d.cbsemigroup}
A \emph{$C_b$-Feller semigroup} is a family $\mk{(T(t), t>0)} \subset \cL(C_b(E), \sigma)$ with the following properties:
\begin{enumerate}
[(i)]
\item $\mk{T(t)}$ is a sub-Markovian kernel operator for every $t > 0$;
\item $\mk{T(t+s)=T(t)T(s)}$ for all $t,s > 0 $;
\item for $f\in C_b(E)$ we have $\mk{T(t)}f \to f$ as $t\to 0$, uniformly on compact subsets of $E$.
\end{enumerate}
\end{defn}

In case that $E$ is locally compact, it \mk{follows along the lines of} \cite[\mk{Lemma} 3.1]{sch98}, 
\mk{which is concerned with the case $E=\Rd$}, 
that a Feller semigroup on $C_0(E)$ can be extended to a 
$C_b$-Feller semigroup on $C_b(E)$. 
We should point out, however, that a $C_b$-Feller semigroup in the above sense does not necessarily leave the space $C_0(E)$ invariant. In that respect, our definition of $C_b$-Feller 
semigroup slightly differs from that in \cite[Definition 4.8.6]{MR1873235} where a $C_b$-Feller semigroup is assumed to be Feller.

Recalling the connection between bp-convergence and $\sigma$-convergence, we see that the requirement (iii) in the above definition in particular implies that $T_tf \to f$ as $t\to 0$ with respect to $\sigma$ and thus, by the semigroup law and the $\sigma$-continuity of the operators $\mk{T(t)}$, that $\mk{T(t)}f \to \mk{T(s)}f$ as $t\downarrow s$ for every $f\in C_b(E)$, i.e.\ the orbits $t\mapsto \mk{T(t)}f$ are right-continuous with respect to $\sigma$. In particular, the orbits have enough measurability to define the Laplace transform of a $C_b$-Feller semigroup by setting
\begin{equation}
\label{eq.sglt}
\langle \mk{R(\lambda)} f, \nu\rangle \coloneqq \int_0^\infty e^{-\lambda t}\langle \mk{T(t)}f, \nu\rangle\,\dx t
\end{equation}
for any $f\in C_b(E)$, $\nu \in \cM(E)$ and $\lambda >0$.

\begin{lem}\label{l.gen}
Let $(\mk{T(t), t>0})$ \mk{be a $C_b$-Feller semigroup}. 
Then, for every $\lambda >0$, Equation \eqref{eq.sglt} defines an operator $\mk{R(\lambda)} \in \cL(C_b(E), \sigma)$. 
Moreover, the family $(\mk{R(\lambda), \lambda >0)}$ consists of injective operators and satisfies the \emph{resolvent identity}
\[
\mk{R(\lambda_1) - R(\lambda_2)} = (\lambda_2-\lambda_1)
\mk{R(\lambda_2)R(\lambda_1)}
\]
for all $\lambda_1, \lambda_2>0$.
\end{lem}

\begin{proof}
By \cite[Theorem 6.2]{kunze11} any $C_b$-Feller semigroup is integrable in the sense of \cite[Definition 5.1]{kunze11}. Now the resolvent identity for the operators $R_\lambda$ follows from \cite[Proposition 5.2]{kunze11}. That the operators $\mk{R(\lambda)}$ are injective is a consequence of \cite[Theorem 2.10]{kunze09}.
\end{proof}

As is \mk{well known}, if $\mk{(R(\lambda), \lambda>0)}$ consists of injective operators and satisfies the resolvent identity, then there exists a unique operator $A$ ($= \lambda - \mk{R(\lambda)}^{-1}$) such that $\mk{R(\lambda)} = (\lambda - A)^{-1}$.

\begin{defn}\label{d.cbgenerator}
Let $(\mk{T(t), t>0})$ be a $C_b$-Feller semigroup. 
The \emph{$C_b$-generator} of $(\mk{T(t), t>0})$ is the unique operator $A$ such that
$\mk{R(\lambda)} = (\lambda - A)^{-1}$ for all $\lambda >0$, where the operators $\mk{R(\lambda)}$ are given by Equation \eqref{eq.sglt}
\kl{, and its domain is}
$\mk{D(A)} \kl{\coloneqq} \mk{\mathrm{rg}R(\lambda)}$, \mk{which is independent of $\lambda>0$.}
\end{defn}

The above gives an ``integral'' definition of the ($C_b$-)generator by means of the Laplace transform of the semigroup. Often, a differential definition of the generator is 
\kb{preferred} and we show next that several differential definitions are in fact equivalent to the above. In one of them, we make use of the so-called \emph{strict topology} $\beta_0$ on $C_b(E)$. This topology is defined as follows: Let $\mathscr{F}_0(E)$ denote the set of functions $\varphi: E\to \R$ that \emph{vanish at infinity}, i.e.\ for every $\eps>0$ there exists a compact set $H\subset E$ with 
$|\varphi (x)|\leq \eps$ for all $x\in E\setminus H$. Then the \emph{strict topology} $\beta_0$ is the locally convex topology generated by the seminorms $\{p_\varphi : \varphi \in \mathscr{F}_0\}$, where $p_\varphi(f) = \|\varphi f\|_\infty$. This topology is consistent with the duality
$(C_b(E), \cM(E))$, i.e., the dual space $(C_b(E), \beta_0)'$ is $\cM(E)$, see \cite[Theorem 7.6.3]{jarchow81}. In fact, it is the Mackey topology of the dual pair $(C_b(E), \cM(E))$, i.e.\ the finest locally convex topology on $C_b(E)$ that yields $\cM(E)$ as a dual space, see \cite[Theorem 4.5 and 5.8]{sentilles72}. This implies that a kernel operator is automatically also $\beta_0$-continuous. By \cite[Theorem 2.10.4]{jarchow81}, $\beta_0$ coincides on $\|\cdot\|_\infty$-bounded subsets on $C_b(E)$ with the topology of uniform convergence on compact subsets of $E$. Thus, condition (iii) in Definition \ref{d.cbsemigroup} can be reformulated by saying $\mk{T(t)}f\to f$ with respect to $\beta_0$ 
as $t\to 0$ for every $f\in C_b(E)$. Taking the $\beta_0$-continuity of the operators $\mk{T(t)}$ into account, it follows that for every $f\in C_b(E)$ the orbit $t\mapsto \mk{T(t)}f$ is 
$\beta_0$ right-continuous.

\begin{thm}\label{t.generatorchar}
Let $(\mk{T(t), t> 0})$ be a  $C_b$-Feller semigroup with $C_b$-generator $A$. Then for $u,f\in C_b(E)$, the following assertions are equivalent.
\begin{enumerate}
[(i)]
\item $u\in D(A)$ and $Au=f$.
\item For every $t>0$ and $x\in E$, we have $\mk{T(t)}u(x) - u(x) = \int_0^t\mk{T(s)}f(x)\,\dx s$.
\item $\sup \{ t^{-1} \|\mk{T(t)}u-u\|_\infty : t\in (0,1)\} <\infty$ and $t^{-1}(\mk{T(t)}u(x) - u(x)) \to f(x)$ as 
$t\to 0$ for all $x\in E$.
\item $t^{-1}(\mk{T(t)}u - u) \to f$ with respect to $\sigma$ as $t\to 0$.
\item $t^{-1}(\mk{T(t)}u-u) \to f$ with respect to 
$\beta_0$ as $t\to 0$.
\item $\sup \{t^{-1}\|\mk{T(t)}u-u\|_\infty : t\in (0,1)\} <\infty$ and $t^{-1}(\mk{T(t)}u-u) \to f$ as $t\to 0$ uniformly on compact subsets of $E$.
\end{enumerate}
\end{thm}

\begin{proof}
(i) $\Rightarrow$ (ii). By \cite[Proposition 5.7]{kunze11}(i), $\langle \mk{T(t)}u - u, \nu\rangle = \int_0^t \langle \mk{T(s)}f, \nu\rangle \,\dx s$
for all $t>0$ and $\nu \in \cM(E)$. 
Picking $\nu = \delta_x$, we get (ii).

(ii) $\Rightarrow$ (iii). We have $t^{-1}(\mk{T(t)}u(x) - u(x)) = t^{-1}\int_0^t \mk{T(s)}f(x)\,\dx s \to 
f(x)$ as $t\to 0$, by the continuity of $s\mapsto \mk{T(s)}f(x)$ in $0$. 
Moreover, 
\[
\|t^{-1}(\mk{T(t)}u(x) - u(x)\|_\infty \leq 
t^{-1} \int_0^t \|\mk{T(s)} f\|_\infty\dx s \leq 
\|f\|_\infty <\infty
\]
for all $t>0$.

(iii) $\Rightarrow$ (iv) follows from the dominated convergence theorem, whereas (iv) $\Rightarrow$ (i) is a consequence of \cite[Theorem 2.10]{kunze09}, applied with $\tau_{\mathfrak{M}} = \sigma$, which corresponds to choosing $\mathfrak{M}$ as the finite subsets of $Y= \cM(E)$.
\smallskip

As $\beta_0$ is the Mackey topology of the pair $(C_b(E),\cM(E))$, we have $\beta = \tau_{\mathfrak{M}}$ where $\mathfrak{M}$ denotes the collection of all absolutely convex subsets of $Y= \cM(E)$ which are $\sigma (\cM(E), C_b(E))$-compact. Thus
the equivalence (i) $\Leftrightarrow$ (v) also follows from \cite[Theorem 2.10]{kunze09}, this time applied with $\tau_\mathfrak{M}= \beta_0$.  The remaining equivalence (v) $\Leftrightarrow$ (vi) follows from the fact that $\beta_0$ coincides with the topology of uniform convergence on compact subset of $E$ on $\|\cdot\|_\infty$-bounded subsets of $C_b(E)$ and the already established implications (v) $\Rightarrow$ (i) $\Rightarrow$ (iii).
\end{proof}

If $(\mk{T(t), t> 0})$ is a $C_b$-Feller semigroup then, by the $\beta_0$-continuity of the operators $\mk{T(t)}$ and (iii) in Definition 
\ref{d.cbsemigroup}, for every $f\in C_b(E)$ the orbit $t\mapsto \mk{T(t)}f$ is right-continuous with respect to $\beta_0$. 
It is a natural question, whether each orbit is actually $\beta_0$-continuous, but, to the best of our knowledge, it is still open. 
However, if $(\mk{T(t), t> 0})$ additionally enjoys the \emph{strong Feller property}, i.e.\ $\mk{T(t)}B_b(E) \subset C_b(E)$ for all $t>0$, then this is indeed the case.

\begin{lem}\label{l.strongfellerorbits}
Let $(\mk{T(t), t> 0})$ be a  $C_b$-Feller semigroup that enjoys the strong Feller property. Then $(\mk{T(t), t> 0})$ has the following additional properties. \mk{Here, in parts (b) and (c) we set $T(0)=I$}.
\begin{enumerate}
[(a)]
\item For every $f\in B_b(E)$, the map $(0,\infty)\times E \ni (t,x) \mapsto \mk{T(t)}f(x)$ is continuous.
\item For every $f\in \mk{C}_b(E)$, the map $[0,\infty) \ni t \mapsto \mk{T(t)}f$ is $\beta_0$-continuous.
\item For every $f \in \mk{C}_b(E)$ and $t_0\in [0,\infty)$, we have $T_tf \to T_{t_0}f$ as $t\to t_0$ uniformly on compact subsets of $E$.
\end{enumerate}
\end{lem}

\begin{proof}
(a) follows from \cite[Proposition V.2.10]{MR850715}. See also \cite[Theorem 3.7]{kunze20}, which shows that the continuity assumption in \cite{MR850715} can be weakened to a measurability assumption.  It follows from (a) and (iii) in Definition \ref{d.cbsemigroup}, that for $f\in C_b(E)$ the map $[0,\infty)\times E \ni (t,x) \mapsto T_tf(x)$ is continuous. Now (b) and (c) follow from \mk{\cite[Theorem 4.4]{kunze09}}.
\end{proof}


\begin{thebibliography}{10}
\bibitem{akk16}
W.~Arendt, S.~Kunkel, and M.~Kunze.
\newblock Diffusion with nonlocal boundary conditions.
\newblock {\em J. Funct. Anal.}, 270(7):2483--2507, 2016.

\bibitem{blm18}
B.~Baeumer, T.~Luks, and M.~M. Meerschaert.
\newblock Space-time fractional {D}irichlet problems.
\newblock {\em Math. Nachr.}, 291(17-18):2516--2535, 2018.

\bibitem{MR3217703}
G.~Barles, E.~Chasseigne, C.~Georgelin, and E.~R. Jakobsen.
\newblock On {N}eumann type problems for nonlocal equations set in a half
  space.
\newblock {\em Trans. Amer. Math. Soc.}, 366(9):4873--4917, 2014.

\bibitem{b-ap07}
I.~Ben-Ari and R.~G. Pinsky.
\newblock Spectral analysis of a family of second-order elliptic operators with
  nonlocal boundary condition indexed by a probability measure.
\newblock {\em J. Funct. Anal.}, 251(1):122--140, 2007.

\bibitem{MR850715}
J.~Bliedtner and W.~Hansen.
\newblock {\em Potential theory}.
\newblock Universitext. Springer-Verlag, Berlin, 1986.
\newblock An analytic and probabilistic approach to balayage.

\bibitem{bobrowski2022concatenation}
A.~Bobrowski.
\newblock Concatenation of dishonest {Feller} processes, exit laws, and limit
  theorems on graphs, 2022.

\bibitem{bogachev}
V.~I. Bogachev.
\newblock {\em Measure theory. {V}ol. {I}, {II}}.
\newblock Springer-Verlag, Berlin, 2007.

\bibitem{MR1438304}
K.~Bogdan.
\newblock The boundary {H}arnack principle for the fractional {L}aplacian.
\newblock {\em Studia Math.}, 123(1):43--80, 1997.

\bibitem{MR2006232}
K.~Bogdan, K.~Burdzy, and Z.-Q. Chen.
\newblock Censored stable processes.
\newblock {\em Probab. Theory Related Fields}, 127(1):89--152, 2003.

\bibitem{MR1671973}
K.~Bogdan and T.~Byczkowski.
\newblock Potential theory for the {$\alpha$}-stable {S}chr\"odinger operator
  on bounded {L}ipschitz domains.
\newblock {\em Studia Math.}, 133(1):53--92, 1999.

\bibitem{MR1703823}
K.~Bogdan and T.~Byczkowski.
\newblock Probabilistic proof of boundary {H}arnack principle for
  {$\alpha$}-harmonic functions.
\newblock {\em Potential Anal.}, 11(2):135--156, 1999.

\bibitem{BOGDAN2019}
K.~Bogdan, T.~Grzywny, K.~Pietruska-Pa\l{}uba, and A.~Rutkowski.
\newblock Extension and trace for nonlocal operators.
\newblock {\em J. Math. Pures Appl.}, 2019.

\bibitem{MR2722789}
K.~Bogdan, T.~Grzywny, and M.~Ryznar.
\newblock Heat kernel estimates for the fractional {L}aplacian with {D}irichlet
  conditions.
\newblock {\em Ann. Probab.}, 38(5):1901--1923, 2010.

\bibitem{MR3165234}
K.~Bogdan, T.~Grzywny, and M.~Ryznar.
\newblock Density and tails of unimodal convolution semigroups.
\newblock {\em J. Funct. Anal.}, 266(6):3543--3571, 2014.

\bibitem{MR3350043}
K.~Bogdan, T.~Grzywny, and M.~Ryznar.
\newblock Barriers, exit time and survival probability for unimodal {L}\'evy
  processes.
\newblock {\em Probab. Theory Related Fields}, 162(1-2):155--198, 2015.

\bibitem{MR2457489}
K.~Bogdan, W.~Hansen, and T.~Jakubowski.
\newblock Time-dependent {S}chr\"odinger perturbations of transition densities.
\newblock {\em Studia Math.}, 189(3):235--254, 2008.

\bibitem{KB-MK-Mp}
K.~Bogdan and M.~Kunze.
\newblock Stable processes with reflections.
\newblock in preparation.

\bibitem{MR3737628}
K.~Bogdan, J.~Rosi\'{n}ski, G.~Serafin, and {\L}.~Wojciechowski.
\newblock L\'{e}vy systems and moment formulas for mixed {P}oisson integrals.
\newblock In {\em Stochastic analysis and related topics}, volume~72 of {\em
  Progr. Probab.}, pages 139--164. Birkh\"{a}user/Springer, Cham, 2017.

\bibitem{MR3295773}
K.~Bogdan and S.~Sydor.
\newblock On nonlocal perturbations of integral kernels.
\newblock In {\em Semigroups of operators---theory and applications}, volume
  113 of {\em Springer Proc. Math. Stat.}, pages 27--42. Springer, Cham, 2015.

\bibitem{MR3156646}
B.~B\"{o}ttcher, R.~Schilling, and J.~Wang.
\newblock {\em L\'{e}vy matters. {III}}, volume 2099 of {\em Lecture Notes in
  Mathematics}.
\newblock Springer, Cham, 2013.
\newblock L\'{e}vy-type processes: construction, approximation and sample path
  properties, With a short biography of Paul L\'{e}vy by Jean Jacod.

\bibitem{MR3160562}
L.~Chaumont, H.~Pant\'{\i}, and V.~Rivero.
\newblock The {L}amperti representation of real-valued self-similar {M}arkov
  processes.
\newblock {\em Bernoulli}, 19(5B):2494--2523, 2013.

\bibitem{MR1473631}
Z.-Q. Chen and R.~Song.
\newblock Intrinsic ultracontractivity and conditional gauge for symmetric
  stable processes.
\newblock {\em J. Funct. Anal.}, 150(1):204--239, 1997.

\bibitem{MR1952456}
Z.-Q. Chen and R.~Song.
\newblock Conditional gauge theorem for non-local {F}eynman-{K}ac transforms.
\newblock {\em Probab. Theory Related Fields}, 125(1):45--72, 2003.

\bibitem{chung86}
K.~L. Chung.
\newblock Doubly-{F}eller process with multiplicative functional.
\newblock In {\em Seminar on stochastic processes, 1985 ({G}ainesville, {F}la.,
  1985)}, volume~12 of {\em Progr. Probab. Statist.}, pages 63--78.
  Birkh\"{a}user Boston, Boston, MA, 1986.

\bibitem{MR2152573}
K.~L. Chung and J.~B. Walsh.
\newblock {\em Markov processes, {B}rownian motion, and time symmetry}, volume
  249 of {\em Grundlehren der mathematischen Wissenschaften [Fundamental
  Principles of Mathematical Sciences]}.
\newblock Springer, New York, second edition, 2005.

\bibitem{MR1329992}
K.~L. Chung and Z.~X. Zhao.
\newblock {\em From {B}rownian motion to {S}chr\"odinger's equation}, volume
  312 of {\em Grundlehren der Mathematischen Wissenschaften}.
\newblock Springer-Verlag, Berlin, 1995.

\bibitem{MR1417491}
G.~Da~Prato and J.~Zabczyk.
\newblock {\em Ergodicity for infinite-dimensional systems}, volume 229 of {\em
  London Mathematical Society Lecture Note Series}.
\newblock Cambridge University Press, Cambridge, 1996.

\bibitem{MR3651008}
S.~Dipierro, X.~Ros-Oton, and E.~Valdinoci.
\newblock Nonlocal problems with {N}eumann boundary conditions.
\newblock {\em Rev. Mat. Iberoam.}, 33(2):377--416, 2017.

\bibitem{ek}
S.~N. Ethier and T.~G. Kurtz.
\newblock {\em Markov processes}.
\newblock Wiley Series in Probability and Mathematical Statistics: Probability
  and Mathematical Statistics. John Wiley \& Sons, Inc., New York, 1986.
\newblock Characterization and convergence.

\bibitem{MR4093464}
M.~R. Evans, S.~N. Majumdar, and G.~Schehr.
\newblock Stochastic resetting and applications.
\newblock {\em J. Phys. A}, 53(19):193001, 67, 2020.

\bibitem{feller-diffusion}
W.~Feller.
\newblock Diffusion processes in one dimension.
\newblock {\em Trans. Amer. Math. Soc.}, 77:1--31, 1954.

\bibitem{MR2266714}
P.~J. Fitzsimmons.
\newblock On the existence of recurrent extensions of self-similar {M}arkov
  processes.
\newblock {\em Electron. Comm. Probab.}, 11:230--241, 2006.

\bibitem{galaskub}
E.~I. Galakhov and A.~L. Skubachevski{\u\i}.
\newblock On {F}eller semigroups generated by elliptic operators with
  integro-differential boundary conditions.
\newblock {\em J. Differential Equations}, 176(2):315--355, 2001.

\bibitem{MR4474283}
P.~Garbaczewski and M.~\.{Z}aba.
\newblock L\'{e}vy processes in bounded domains: path-wise reflection scenarios
  and signatures of confinement.
\newblock {\em J. Phys. A}, 55(30):Paper No. 305005, 26, 2022.

\bibitem{hairer21}
M.~Hairer.
\newblock Convergence of {M}arkov processes.
\newblock Minicourse available at http://www.hairer.org/notes/Convergence.pdf,
  2021.

\bibitem{MR202197}
N.~Ikeda, M.~Nagasawa, and S.~Watanabe.
\newblock A construction of {M}arkov processes by piecing out.
\newblock {\em Proc. Japan Acad.}, 42:370--375, 1966.

\bibitem{MR4514832}
A.~Iksanov and A.~Pilipenko.
\newblock On a skew stable {L}\'{e}vy process.
\newblock {\em Stochastic Process. Appl.}, 156:44--68, 2023.

\bibitem{MR1873235}
N.~Jacob.
\newblock {\em Pseudo differential operators and {M}arkov processes. {V}ol.
  {I}}.
\newblock Imperial College Press, London, 2001.
\newblock Fourier analysis and semigroups.

\bibitem{jarchow81}
H.~Jarchow.
\newblock {\em Locally convex spaces}.
\newblock B. G. Teubner, Stuttgart, 1981.
\newblock Mathematische Leitf\"{a}den. [Mathematical Textbooks].

\bibitem{MR4520527}
P.~Kim, R.~Song, and Z.~Vondra\v{c}ek.
\newblock Positive self-similar {M}arkov processes obtained by resurrection.
\newblock {\em Stochastic Process. Appl.}, 156:379--420, 2023.

\bibitem{kim2022potential}
P.~Kim, R.~Song, and Z.~Vondraček.
\newblock Potential theory of dirichlet forms with jump kernels blowing up at
  the boundary, 2022.

\bibitem{MR1162571}
T.~Komorowski.
\newblock Asymptotic periodicity of some stochastically perturbed dynamical
  systems.
\newblock {\em Ann. Inst. H. Poincar\'{e} Probab. Statist.}, 28(2):165--178,
  1992.

\bibitem{krein82}
S.~G. Kre{\u\i}n.
\newblock {\em Linear equations in {B}anach spaces}.
\newblock Birkh\"auser, Boston, Mass., 1982.
\newblock Translated from the Russian by A. Iacob, With an introduction by I.
  Gohberg.

\bibitem{kulik15}
A.~Kulik.
\newblock {\em Introduction to ergodic rates for {M}arkov chains and
  processes}, volume~2 of {\em Lectures in Pure and Applied Mathematics}.
\newblock Potsdam University Press, Potsdam, 2015.
\newblock With applications to limit theorems.

\bibitem{kunze09}
M.~Kunze.
\newblock Continuity and equicontinuity of semigroups on norming dual pairs.
\newblock {\em Semigroup Forum}, 79(3):540--560, 2009.

\bibitem{kunze11}
M.~Kunze.
\newblock A {P}ettis-type integral and applications to transition semigroups.
\newblock {\em Czechoslovak Math. J.}, 61(136)(2):437--459, 2011.

\bibitem{kunze20}
M.~Kunze.
\newblock Diffusion with nonlocal {D}irichlet boundary conditions on unbounded
  domains.
\newblock {\em Studia Math.}, 253(1):1--38, 2020.

\bibitem{MR3613319}
M.~Kwa\'{s}nicki.
\newblock Ten equivalent definitions of the fractional {L}aplace operator.
\newblock {\em Fract. Calc. Appl. Anal.}, 20(1):7--51, 2017.

\bibitem{MR1452617}
A.~Lasota.
\newblock From fractals to stochastic differential equations.
\newblock In {\em Chaos---the interplay between stochastic and deterministic
  behaviour ({K}arpacz, 1995)}, volume 457 of {\em Lecture Notes in Phys.},
  pages 235--255. Springer, Berlin, 1995.

\bibitem{MR1244104}
A.~Lasota and M.~C. Mackey.
\newblock {\em Chaos, fractals, and noise}, volume~97 of {\em Applied
  Mathematical Sciences}.
\newblock Springer-Verlag, New York, second edition, 1994.
\newblock Stochastic aspects of dynamics.

\bibitem{MR1265226}
A.~Lasota and J.~A. Yorke.
\newblock Lower bound technique for {M}arkov operators and iterated function
  systems.
\newblock {\em Random Comput. Dynam.}, 2(1):41--77, 1994.

\bibitem{meyer75}
P.~A. Meyer.
\newblock Renaissance, recollements, m\'{e}langes, ralentissement de processus
  de {M}arkov.
\newblock {\em Ann. Inst. Fourier (Grenoble)}, 25(3-4):xxiii, 465--497, 1975.

\bibitem{MR4140082}
H.~Pant\'{\i}, J.~C. Pardo, and V.~M. Rivero.
\newblock Recurrent extensions of real-valued self-similar {M}arkov processes.
\newblock {\em Potential Anal.}, 53(3):899--920, 2020.

\bibitem{MR2273672}
P.~E. Protter.
\newblock {\em Stochastic integration and differential equations}, volume~21 of
  {\em Stochastic Modelling and Applied Probability}.
\newblock Springer-Verlag, Berlin, 2005.
\newblock Second edition. Version 2.1, Corrected third printing.

\bibitem{revuz}
D.~Revuz.
\newblock {\em Markov chains}.
\newblock North-Holland Publishing Co., Amsterdam, 1975.
\newblock North-Holland Mathematical Library, Vol. 11.

\bibitem{MR2364226}
V.~Rivero.
\newblock Recurrent extensions of self-similar {M}arkov processes and
  {C}ram\'{e}r's condition. {II}.
\newblock {\em Bernoulli}, 13(4):1053--1070, 2007.

\bibitem{MR1739520}
K.~Sato.
\newblock {\em L\'evy processes and infinitely divisible distributions},
  volume~68 of {\em Cambridge Studies in Advanced Mathematics}.
\newblock Cambridge University Press, Cambridge, 1999.
\newblock Translated from the 1990 Japanese original, Revised by the author.

\bibitem{schaefer}
H.~H. Schaefer.
\newblock {\em Banach lattices and positive operators}.
\newblock Die Grundlehren der mathematischen Wissenschaften, Band 215.
  Springer-Verlag, New York-Heidelberg, 1974.

\bibitem{sch98}
R.~L. Schilling.
\newblock Conservativeness and extensions of {F}eller semigroups.
\newblock {\em Positivity}, 2(3):239--256, 1998.

\bibitem{MR3644418}
R.~L. Schilling.
\newblock {\em Measures, integrals and martingales}.
\newblock Cambridge University Press, Cambridge, second edition, 2017.

\bibitem{sentilles72}
F.~D. Sentilles.
\newblock Bounded continuous functions on a completely regular space.
\newblock {\em Trans. Amer. Math. Soc.}, 168:311--336, 1972.

\bibitem{MR958914}
M.~Sharpe.
\newblock {\em General theory of {M}arkov processes}, volume 133 of {\em Pure
  and Applied Mathematics}.
\newblock Academic Press, Inc., Boston, MA, 1988.

\bibitem{MR4378848}
A.~A. Stanislavsky and A.~Weron.
\newblock Subdiffusive search with home returns via stochastic resetting: a
  subordination scheme approach.
\newblock {\em J. Phys. A}, 55(7):Paper No. 074004, 15, 2022.

\bibitem{taira}
K.~Taira.
\newblock {\em Semigroups, boundary value problems and {M}arkov processes}.
\newblock Springer Monographs in Mathematics. Springer-Verlag, Berlin, 2004.

\bibitem{MR4245573}
Z.~Vondra\v{c}ek.
\newblock A probabilistic approach to a non-local quadratic form and its
  connection to the {N}eumann boundary condition problem.
\newblock {\em Math. Nachr.}, 294(1):177--194, 2021.

\bibitem{MR4247975}
F.~Werner.
\newblock Concatenation and pasting of right processes.
\newblock {\em Electron. J. Probab.}, 26:Paper No. 50, 21, 2021.
\end{thebibliography}
\end{document}